\documentclass[11pt,letterpaper]{article}
\usepackage[english]{babel}
\usepackage{amsmath}
\usepackage{amsfonts, mathrsfs}
\usepackage{amssymb}
\usepackage{verbatim}
\usepackage{graphicx}
\usepackage{color}
\usepackage[numbers]{natbib}
\usepackage{url}
\usepackage{enumitem}

\numberwithin{equation}{section}
\newtheorem{thm}{Theorem}[section]

\newtheorem{prop}[thm]{Proposition}
\newtheorem{lemma}[thm]{Lemma}
\newtheorem{examples}[thm]{Examples}
\newtheorem{assum}[thm]{Assumption}
\newtheorem{defn}[thm]{Definition}

\newtheorem{remarks}[thm]{Remarks}
\newenvironment{proof}[1][Proof]{\textbf{#1.} }{\ \rule{0.5em}{0.5em}}

\newcommand{\var}{{\rm Var} \mspace{1mu}}

\DeclareMathOperator{\arctanh}{arctanh}

\renewcommand{\labelenumi}{\alph{enumi}.)}

\begin{document}

\title{Convergence of the Population Dynamics algorithm \\ in the Wasserstein metric}

\author{Mariana Olvera-Cravioto}

\maketitle

\begin{abstract}

We study the convergence of the population dynamics algorithm, which produces sample pools of random variables having a distribution that closely approximates that of the {\em special endogenous solution} to a stochastic fixed-point equation of the form:
$$R\stackrel{\mathcal D}{=} \Phi( Q, N, \{ C_i \}, \{R_i\}),$$
where $(Q, N, \{C_i\})$ is a real-valued random vector with $N \in \mathbb{N}$, and $\{R_i\}_{i \in \mathbb{N}}$ is a sequence of i.i.d.~copies of $R$, independent of $(Q, N, \{C_i\})$; the symbol $\stackrel{\mathcal{D}}{=}$ denotes equality in distribution. Specifically, we show its convergence in the Wasserstein metric of order $p$ ($p \geq 1$) and prove the consistency of estimators based on the sample pool produced by the algorithm.  

\end{abstract}

\section{Introduction}
\label{sec:intro}

We study an iterative bootstrap algorithm, known as the ``population dynamics" algorithm, that can be used to efficiently generate samples of random variables whose distribution closely approximates that of the so-called special endogenous solution to a stochastic fixed-point equation (SFPE) of the form:
\begin{equation}\label{eq:SFPE}
    R\stackrel{\mathcal D}{=} \Phi( Q, N, \{ C_i \}, \{R_i\}) ,
\end{equation}
where $(Q, N, \{C_i\})$ is a real-valued random vector with $N \in \mathbb{N} = \{0, 1, 2, \dots\}$, and $\{R_i\}_{i \in \mathbb{N}}$ is a sequence of i.i.d. copies of $R$, independent of $(Q, N, \{C_i\})$. These equations appear in a variety of problems, ranging from computer science to statistical physics, e.g.: in the analysis of divide and conquer algorithms such as Quicksort \cite{Rosler_91,Fill_Jan_01,Ros_Rus_01} and {\tt FIND}  \cite{Devroye_01}, the analysis of Google's PageRank algorithm \cite{Volk_Litv_08,Jel_Olv_10, Chen_Litv_Olv_14, Lee_Olv_17}, the study of queueing networks with synchronization requirements \cite{Kar_Kel_Suh_94, Olv_Ruiz_14}, and the analysis of the Ising model \cite{Dem_Mon_10}, to name a few. In general, SPFEs of the form in \eqref{eq:SFPE} can have multiple solutions, but in most cases we are interested in computing those that can be explicitly constructed on a weighted branching process, known as {\em endogenous} solutions. In some cases, even the endogenous solution is not unique \cite{Alsm_Mein_10b}, but characterizing all endogenous solutions can be done using the {\em special endogenous} solution, which is the only attracting solution, and can be constructed by iterating \eqref{eq:SFPE} starting from some well-behaved initial distribution. 

This work focuses on the analysis of a simulation algorithm that can be used to generate samples from a distribution that closely approximates that of the special endogenous solution to a variety of SFPEs. The need for such an approximate algorithm lies on the numerical complexity of simulating even a few generations of a weighted branching process using naive Monte Carlo methods. The population dynamics algorithm, described in \S14.6.4 in \cite{Mezard_Montanari_2009} and \S8.1 in \cite{Aldo_Band_05}, circumvents this problem by resampling with replacement from previously computed iterations of \eqref{eq:SFPE}, i.e., by using an iterative bootstrap technique. However, as is the case with the standard bootstrap algorithm, the samples obtained are neither independent nor exactly distributed according to the target distribution, which raises the need to study the convergence properties of the algorithm.

Before presenting the algorithm and stating our main results, it may be helpful to describe in more detail some of the examples mentioned above. Throughout the paper, we use $x \vee y = \max\{x, y\}$  and $x \wedge y = \min\{x,y\}$ to denote the maximum and the minimum, respectively, of $x$ and $y$. 
\begin{itemize}
\item The linear SFPE or ``smoothing transform": 
\begin{equation} \label{eq:Linear}
R\stackrel{\mathcal D}{=} Q + \sum_{i=1}^{N} C_i R_i,
\end{equation}
appears in the analysis of the number of comparisons required by the sorting algorithm Quicksort \cite{Rosler_91,Fill_Jan_01,Ros_Rus_01}, and can also be used to describe the distribution of the ranks computed by Google's PageRank algorithm on directed complex networks \cite{Volk_Litv_08,Jel_Olv_10,Chen_Litv_Olv_14, Lee_Olv_17}.

\item The maximum SFPE or ``high-order Lindley equation": 
\begin{equation}\label{eq:Max}
    R\stackrel{\mathcal D}{=} Q \vee \bigvee_{i=1}^{N} C_i R_i, \qquad \text{equivalently,} \qquad X \stackrel{\mathcal D}{=} T \vee \bigvee_{i=1}^N ( \xi_i + X_i),
\end{equation}
arises as the limiting waiting time distribution on queueing networks with parallel servers and synchronization requirements \cite{Kar_Kel_Suh_94, Olv_Ruiz_14} and in the analysis of the branching random walk \cite{Aldo_Band_05}.

\item The discounted tree-sum SFPE: 
\begin{equation} \label{eq:DiscountedTrees}
R \stackrel{\mathcal{D}}{=} Q + \bigvee_{i=1}^N C_i R_i
\end{equation}
appears in the  worst-case analysis of the {\tt FIND} algorithm \cite{Devroye_01} and the analysis of the ``discounted branching random walk" \cite{Athreya_85}.

\item The ``free-entropy" SFPE:
\begin{equation} \label{eq:IsingModel}
R \stackrel{\mathcal{D}}{=} Q + \sum_{i=1}^{N} \arctanh ( \tanh(\beta) \tanh( R_i))
\end{equation}
characterizes the asymptotic free-entropy density in the ferromagnetic Ising model on locally tree-like graphs \cite{Dem_Mon_10}. In this case, $C_i \equiv \tanh(\beta)$ for all $i \geq 1$,  $\beta \geq 0$ represents the ``inverse temperature", and $Q$  the magnetic field. 

\item Although the analysis presented here does not directly apply to this case, we mention that the population dynamics algorithm can also be used to simulate the fixed points of the belief propagation equations on random graphical models \cite{Mezard_Montanari_2009}: 
\begin{equation*}
R \stackrel{\mathcal{D}}{=} \Phi\left( Q, N, \{ C_i\}, \{\tilde R_i\} \right) \qquad \text{and} \qquad \tilde R \stackrel{\mathcal{D}}{=} \Psi\left( \tilde Q, \tilde N, \{ \tilde C_i\}, \{ R_i\} \right), 
\end{equation*}
where the $\{\tilde R_i\}$ are i.i.d.~copies of $\tilde R$ independent of the vector $(Q, N, \{C_i\})$ and the $\{R_i\}$ are i.i.d.~copies of $R$ independent of the vector $(\tilde Q, \tilde N, \{ \tilde C_i \})$, with $\Phi$ and $\Psi$ potentially different. 
\end{itemize}
We refer the reader to  \cite{Aldo_Band_05} for even more examples, including some involving minimums. 

The existence and uniqueness of solutions to any of these SFPEs is in itself a non-trivial problem. We refer the reader again to \cite{Aldo_Band_05} for a broad survey of known results and open problems on this topic. The most well-studied equations are the linear \eqref{eq:Linear} and maximum \eqref{eq:Max} SFPEs, which have been extensively studied in \cite{Liu_98, Iksanov_04, Als_Big_Mei_10, Alsm_Mein_10a, Alsm_Mein_10b, Als_Dysz_17, Janson_15} and \cite{Biggins_98, Jel_Olv_15}, respectively.  However, to provide some context to where the population dynamics algorithm fits in, we briefly mention that the existence of solutions is often established by showing that the transformation $T$ that maps the distribution $\mu$ on $\mathbb{R}$ to the distribution of 
$$\Phi\left( Q, N, \{ C_i\}, \{ X_i\} \right),$$
where the $\{X_i\}$ are i.i.d.~random variables distributed according to $\mu$, independent of the vector $(Q, N, \{C_i\})$, is strictly contracting under some suitable metric. Note that in this case, we have that the sequence of probability measures $\mu_{n+1} = T(\mu_n)$ converges as $n \to \infty$ to a fixed point of \eqref{eq:SFPE}.  Moreover, as long as the initial distribution $\mu_0$ has sufficiently light tails, one can show that $\{\mu_n\}$ converges to the special endogenous solution to \eqref{eq:SFPE}, and the contracting nature of $T$ provides an upper bound of the form 
$$d(\mu_n, \mu) \leq d(T(\mu_{n-1}), T(\mu)) \leq c d(\mu_{n-1}, \mu) \leq c^n d(\mu_0,\mu), \qquad n = 1, 2, \dots,$$
for some constant $0 < c < 1$, where $d$ is the distance under which $T$ is a contraction.

As will be discussed in more detail later (see Examples~\ref{E.Examples}), all the examples provided earlier define contractions under $d_p$, the Wasserstein metric of order $p$, for some $p \geq 1$. For completeness, we also include a result (Theorem~\ref{T.LpConvergence}) that gives easy to verify conditions guaranteeing that 
$$E\left[ \left| R^{(k)} - R \right|^\beta \right] \leq c^k$$
for some $0 < c < 1$, where $R^{(k)}$ and $R$ have distributions $\mu_k$ and $\mu$, respectively.

It follows that from a computational point of view, it suffices to have an algorithm for computing $\mu_k$ for a fixed number of iterations $k \in \mathbb{N}$.  The population dynamics algorithm produces a sample of observations approximately distributed according to $\mu_k$, which can also be helpful in searching for the existence of endogenous solutions, as stated in \cite{Aldo_Band_05}. We now describe how to obtain an exact sample of $\mu_k$, which will also make clear the need for a computationally efficient method.

\subsection{Constructing endogenous solutions on weighted branching processes} \label{SS.EndogenousSol}

As mentioned earlier, the attracting endogenous solution to \eqref{eq:SFPE}, provided it exists, can be constructed on a structure known as a weighted branching process \cite{Rosler_91}. We now elaborate on this point.

Let $\mathbb{N}_+ = \{1, 2, 3, \dots\}$ be the set of positive integers and let $U = \bigcup_{k=0}^\infty (\mathbb{N}_+)^k$ be the set of all finite sequences ${\bf i} = (i_1, i_2, \dots, i_n)$, $n\ge 0$, where by convention $\mathbb{N}_+^0 = \{ \emptyset\}$ contains the null sequence $\emptyset$. To ease the exposition, we will use $({\bf i}, j) = (i_1,\dots, i_n, j)$ to denote the index concatenation operation. Next, let $(Q, N, \{C_i\}_{i\geq 1})$ be a real-valued vector with $N \in \mathbb{N}$. We will refer to this vector as the generic branching vector. Now let $\{ (Q_{\bf i}, N_{\bf i}, \{C_{({\bf i}, j)}\}_{j\geq 1} \}_{{\bf i} \in U}$ be a sequence of i.i.d.~copies of the generic branching vector. To construct a weighted branching process we start by defining a tree as follows: let $A_0 = \{ \emptyset\}$ denote the root of the tree, and define the $n$th generation according to the recursion
$$A_n = \{ ({\bf i}, i_n) \in U: {\bf i} \in A_{n-1}, 1 \leq i_n \leq N_{\bf i} \}, \quad n \geq 1.$$
Now, assign to each node ${\bf i}$ in the tree a weight $\Pi_{\bf i}$ according to the recursion
$$\Pi_\emptyset \equiv 1, \qquad \Pi_{({\bf i}, i_n)} = C_{({\bf i}, i_n)} \Pi_{\bf i}, \qquad n \geq 1, $$
see Figure \ref{F.Tree}. If $P(N< \infty)=1$ and $C_i \equiv 1$ for all $i \geq 1$, the weighted branching process reduces to a Galton-Watson process.

\begin{figure}[t]
\centering
\begin{picture}(350,110)(0,0)
\put(27,8){\includegraphics[scale = 0.75]{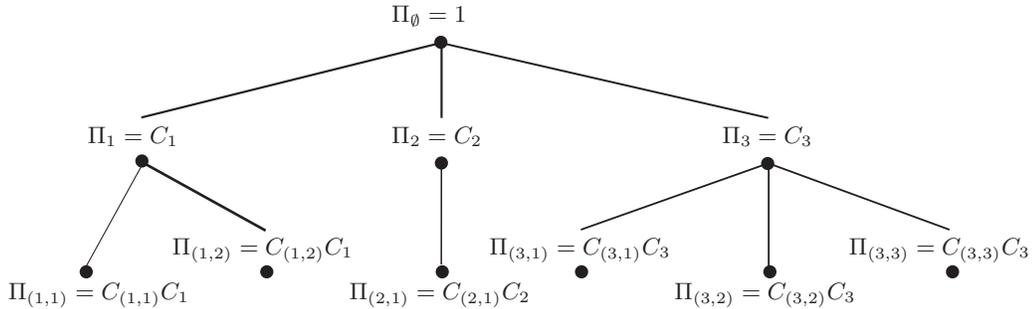}}
\put(145,105){\footnotesize $\Pi_\emptyset = 1$}
\put(30,59){\footnotesize $\Pi_{1} = C_1$}
\put(145,59){\footnotesize $\Pi_{2} = C_2$}
\put(270,59){\footnotesize $\Pi_{3} = C_3$}
\put(0,0){\footnotesize $\Pi_{(1,1)} = C_{(1,1)} C_1$}
\put(62,17){\footnotesize $\Pi_{(1,2)} = C_{(1,2)} C_1$}
\put(129,0){\footnotesize $\Pi_{(2,1)} = C_{(2,1)} C_2$}
\put(182,17){\footnotesize $\Pi_{(3,1)} = C_{(3,1)} C_3$}
\put(252,0){\footnotesize $\Pi_{(3,2)} = C_{(3,2)} C_3$}
\put(318,17){\footnotesize $\Pi_{(3,3)} = C_{(3,3)} C_3$}
\end{picture}
\caption{Weighted branching process}\label{F.Tree}
\end{figure}

To generate a sample from $\mu_k$ we first need to fix the initial distribution $\mu_0$, e.g., by letting $\mu_0$ be the probability measure of a constant, say zero or one. Now construct a weighted branching process with $k$ generations, and let $\{R^{(0)}_{\bf i} \}_{{\bf i} \in A_k}$ be i.i.d.~random variables having distribution $\mu_0$. Next, define recursively for each ${\bf i} \in A_{k-r}$, $1 \leq r \leq k$, 
$$R^{(r)}_{\bf i} = \Phi\left( Q_{\bf i}, N_{\bf i}, \{ C_{({\bf i},j)}\}_{j \geq 1}, \{ R^{(r-1)}_{({\bf i},j)} \}_{j \geq 1} \right).$$
The random variable $R^{(k)}_\emptyset$ is distributed according to $\mu_k$, and its generation requires on average $(E[N])^k$ i.i.d.~copies of the generic branching vector $(Q, N, \{C_i\}_{i\geq 1})$. It follows that if the goal was to obtain an i.i.d.~sample of size $m$ from distribution $\mu_k$, one would need to generate on average $m (E[N])^k$ copies of the generic branching vector. However, in applications one typically has $E[N] > 1$, e.g., $N \equiv 2$ for Quicksort, $E[N] \approx 30$ in the analysis of PageRank on the WWW graph, and $E[N]$ can be in the hundreds for MapReduce implementations related to the maximum SFPE. This makes the exact simulation of $R^{(k)}$ using a weighted branching process impractical. 

The population dynamics algorithm, described below, uses a bootstrap approach to produce a sample of size $m$ of random variables that are approximately distributed according to $\mu_k$, and that although not independent, can be used to obtain consistent estimators for moments, quantiles and other functions of $\mu_k$.

\subsection{The population dynamics algorithm}

The population dynamics algorithm is based on the bootstrap, i.e., in the idea of sampling with replacement random variables from a common pool.  As described above, the algorithm starts by generating a sample of i.i.d.~random variables having distribution $\mu_0$, with the difference that when computing the next level of the recursion, it samples with replacement from this pool as needed by the map $\Phi$. In other words,  to obtain a pool of approximate copies of $R^{(j)}$ we bootstrap from the pool previously obtained of approximate copies of $R^{(j-1)}$. The approximation lies in the fact that we are not sampling from $R^{(j-1)}$ itself, but from a finite sample of conditionally independent observations that are only approximately distributed as $R^{(j-1)}$. The algorithm is described below.

Let $(Q,N, \{C_r\})$ denote the generic branching vector defining the weighted branching process. Let $k$ be the depth of the recursion that we want to simulate, i.e., the algorithm will produce a sample of random variables approximately distributed according to $\mu_k$. Choose $m\in \mathbb N_+$ to be the bootstrap sample size. For each $0 \leq j \leq k$, the algorithm outputs $\mathscr{P}^{(j,m)} \triangleq \left(\hat R^{(j,m)}_1, \hat R^{(j,m)}_2,\dots,\hat R^{(j,m)}_m\right)$, which we refer to as the sample pool at level $j$.
\begin{enumerate}
 \item \emph{Initialize}: Set $j=0$. Simulate a sequence $\{R^{(0)}_i\}_{i = 1}^m$ of i.i.d.~random variables distributed according to some initial distribution $\mu_0$. Let $\hat R^{(0,m)}_i = R^{(0)}_i$ for $i = 1, \dots, m$. \\
 Output  $\mathscr{P}^{(0,m)}=\left(\hat R^{(0,m)}_1, \hat R^{(0,m)}_2,\dots,\hat R^{(0,m)}_m\right)$ and update $j = 1$.

\item While $j \leq k$:

\begin{enumerate}
\item Simulate a sequence $\{ (Q_i^{(j)}, N_i^{(j)}, \{ C_{(i,r)}^{(j)} \}_{r \geq 1} ) \}_{i =1}^m$ of i.i.d. copies of the generic branching vector, independent of everything else.

\item Let
        \begin{equation}
            \hat R^{(j,m)}_i = \Phi\left( Q_i^{(j)}, N_i^{(j)}, \{ C_{(i,r)}^{(j)} \}, \{  \hat R^{(j-1,m)}_{(i,r)}\} \right), \qquad i=1,\dots,m,\label{eq:bootstrap-recur}
        \end{equation}
        where the $\hat R^{(j-1,m)}_{(i,r)}$ are sampled uniformly with replacement from the pool $\mathscr{P}^{(j-1,m)}$.

\item Output $\mathscr{P}^{(j,m)}=\left(\hat R^{(j,m)}_1, \hat R^{(j,m)}_2,\dots,\hat R^{(j,m)}_m\right)$ and update $j = j+1$.

\end{enumerate}

\end{enumerate}

    We conclude this section by pointing out that the complexity of the algorithm described above is of order $k m$, while the naive Monte Carlo approach described earlier, which consists on sampling $m$ i.i.d. copies of a weighted branching process up to the $k$th generation, has order $(E[N])^k m$.  Our main results establish the convergence of the algorithm in the Wasserstein metric of order $p$ ($p \geq 1$), as well as the consistency of estimators constructed using the pool $\mathscr{P}^{(k,m)}$. The following section contains all the statements, and the proofs are given in Section~\ref{S.Proofs}.

\section{Main results} 

We start by defining the Wasserstein metric of order $p$. 

\begin{defn} \label{d.wasserstein}
    Let $M(\mu, \nu)$ denote the set of joint probability measures on $\mathbb R \times \mathbb R$ with marginals $\mu$ and $\nu$. Then, the Wasserstein metric of order $p$ ($1 \leq p < \infty$) between $\mu$ and $\nu$ is given by
$$d_p(\mu, \nu) = \inf_{\pi \in M(\mu, \nu)} \left(  \int_{\mathbb R\times \mathbb R} | x - y |^p \, d \pi(x, y) \right)^{1/p}.$$
\end{defn}

An important advantage of working with the Wasserstein metrics is that on the real line they admit the explicit representation 
  \begin{equation} \label{eq:Explicit}
  d_p(\mu, \nu) = \left( \int_{0}^1 | F^{-1}(u) - G^{-1}(u) |^p du \right)^{1/p} ,
  \end{equation}
where $F$ and $G$ are the cumulative distribution functions of $\mu$ and $\nu$, respectively, and $f^{-1}(t) = \inf\{ x \in \mathbb{R}: f(x) \geq t\}$ denotes the generalized inverse of $f$. It follows that the optimal coupling of two real random variables $X$ and $Y$ is given by $(X, Y) = (F^{-1}(U), G^{-1}(U))$, where $U$ is uniformly distributed in $[0, 1]$.

With some abuse of notation, we use $d_p(F,G)$ to denote the Wasserstein distance of order $p$ between the probability measures $\mu$ and $\nu$, where $F(x) = \mu((-\infty, x])$ and $G(x) = \nu((-\infty, x])$ are their corresponding cumulative distribution functions.

Our main results establish the convergence of $d_p(\hat F_{k,m}, F_k)$ as $m \to \infty$, both in mean and almost surely, where 
$$\hat F_{k,m}(x) = \frac{1}{m} \sum_{i=1}^m 1(\hat R_i^{(k,m)} \leq x) \qquad \text{and} \qquad F_k(x) = \mu_k((-\infty, x]), \qquad k \in \mathbb{N},$$
and $\mathscr{P}^{(k,m)} = \left( \hat R_1^{(k,m)}, \dots, \hat R_m^{(k,m)} \right)$ is the pool generated by the population dynamics algorithm.  The theorems are proven under two different assumptions, the first one imposing a Lipschitz condition on the mean of $\Phi$, and the second one requiring $\Phi$ to be Lipschitz continuous almost surely. 

\begin{assum} \label{A.PhiAssumption} 
For some $p \geq 1$ there exist a constant $0 < H_p < \infty$ such that if $\{ (X_i, Y_i) : i \geq 1 \}$ is a sequence of i.i.d.~random vectors, independent of $(Q, N, \{C_r\})$, then
$$E\left[ \left| \Phi(Q, N, \{C_r\}, \{X_r\}) - \Phi(Q, N, \{C_r\}, \{Y_r\}) \right|^p \right] \leq H_p E[ |X_1 - Y_1|^p] .$$
{\sc [linear0]} For the linear SFPE \eqref{eq:Linear}, it suffices that the inequality holds for $\{ X_i \}$ and $\{ Y_i\}$ having the same mean. 
\end{assum} 

\begin{assum} \label{A.Lipschitz}
Suppose that for any vector $(q, n, \{c_r\})$, with $n \in \mathbb{N} \cup \{\infty\}$, and any sequences of numbers $\{ x_r\}$ and $\{y_r\}$ for which $\Phi(q, n, \{c_r\}, \{x_r\})$ and $\Phi(q, n, \{c_r\}, \{x_r\})$ are well defined, there exists a function $\varphi:\mathbb{R} \to \mathbb{R}_+$ such that
$$\left| \Phi(q, n, \{c_r\}, \{x_r\}) - \Phi(q, n, \{c_r\}, \{y_r\}) \right| \leq \sum_{r=1}^n \varphi(c_r) |x_r - y_r|.$$
\end{assum}

\begin{remarks} \label{R.Assumptions}
\begin{enumerate} \renewcommand{\labelenumi}{(\roman{enumi})}

\item To see that Assumption~\ref{A.Lipschitz} implies Assumption~\ref{A.PhiAssumption}, note that Lemma~4.1 in \cite{Jel_Olv_12a} gives that
\begin{align*}
E\left[ \left( \sum_{r=1}^N \varphi(C_r) |X_r - Y_r| \right)^p \right] &\leq  E\left[ \sum_{r=1}^N \varphi(C_r)^p | X_r - Y_r|^p \right]  \\
 &\hspace{5mm} + E\left[ \left( \sum_{r=1}^N \varphi(C_r)  \right)^p \right] \left( E[ |X_1 - Y_1|^{\lceil p \rceil -1}] \right)^{p/(\lceil p \rceil-1)} \\
&\leq 2 E\left[ \left( \sum_{r=1}^N \varphi(C_r)  \right)^p \right] E\left[ |X_1 - Y_1|^p \right],
\end{align*}
and therefore Assumption~\ref{A.PhiAssumption} holds with $H_p = 2E\left[ \left( \sum_{r=1}^N \varphi(C_r)  \right)^p \right]$, provided the expectation is finite. However, much tighter bounds can be obtained for specific examples, and we can usually find $p \geq 1$ such that $H_p < 1$.

\item  The existence of a $p \geq 1$ for which $H_p < 1$ is important for obtaining estimates for the rate of convergence of the algorithm that are uniform in $k$, and has also important implications for the convergence of $R^{(k)} \to R$ as $k \to \infty$, as the next result shows. 

\end{enumerate}
\end{remarks}

\begin{thm} \label{T.LpConvergence}
Suppose Assumption~\ref{A.PhiAssumption} holds for some $p \geq 1$, $H_p < 1$, and any i.i.d.~sequence $\{(X_i, Y_i): i \geq 1\}$ independent of $(Q,N, \{C_r\})$. Then, provided $E\left[ |R^{(0)}|^p + | \Phi(Q, N, \{C_r\}, \{ 0\})|^p \right] < \infty$, there exists a random variable $R$ and constants $0 \leq c_p < 1$ and $A_p <\infty$ such that
\begin{equation} \label{eq:GeomConv}
E\left[ \left| R^{(k)} - R \right|^p \right] \leq A_p c_p^k \to 0, \qquad k \to \infty,
\end{equation}
where $R^{(k)}$ and $R$ are distributed according to $\mu_k$ and $\mu$, respectively. For the linear SFPE \eqref{eq:Linear}, we have that \eqref{eq:GeomConv} also holds under either of the following conditions:
\begin{itemize}[leftmargin=15pt]
\item[i)] If Assumption~\ref{A.PhiAssumption} {\sc [linear0]} holds and $E[Q] = E[R^{(0)}] = 0$.
\item[ii)] If $E\left[ \left( \sum_{i=1}^N |C_i| \right)^p + |R^{(0)} |^p + |Q|^p \right] < \infty$ and $\rho_1 \vee \rho_p < 1$, where $\rho_\beta \triangleq E\left[ \sum_{i=1}^N |C_i|^\beta \right]$.
\end{itemize}
\end{thm}

As the proof of Theorem~\ref{T.LpConvergence} shows, one can take $c_p = H_p$ under the main set of conditions as well as under conditions (i), whereas for (ii) we have $c_p = \rho_1 \vee \rho_p$.  As a consequence of the proof of Theorem~\ref{T.LpConvergence} we also obtain the following explicit bound for the moments of $R^{(k)}$.

\begin{lemma} \label{L.Moments}
Suppose Assumption~\ref{A.PhiAssumption} holds for some $p \geq 1$. In the linear case, if only Assumption~\ref{A.PhiAssumption} {\sc [linear0]} holds, suppose further that $E[ R^{(0)}] = E[Q] = 0$. Then, for any $k \geq 0$,
$$\left( E[ |R^{(k)}|^p ] \right)^{1/p} \leq A_p \sum_{i=0}^{k-1} (H_p^{1/p})^i ,$$
where $A_p = (H_p^{1/p}+1) \left( E[ |R^{(0)}|^p] \right)^{1/p} + \left(  E\left[  \left|  \Phi(Q, N, \{C_r\}, \{ 0 \})  \right|^p \right] \right)^{1/p} $. 
\end{lemma}

Before stating the main theorems establishing the convergence of the algorithm in the Wasserstein metric, we point out how Assumptions~\ref{A.PhiAssumption} and \ref{A.Lipschitz} are satisfied by all the examples mentioned in the introduction.

\begin{examples} \label{E.Examples}
\begin{itemize}
\item The linear SFPE \eqref{eq:Linear} clearly satisfies Assumption~\ref{A.Lipschitz} with $\varphi(t) = |t|$. Moreover, for the Quicksort algorithm studied in \cite{Rosler_91,Fill_Jan_01,Ros_Rus_01} we have $N \equiv 2$, $C_1 = U = 1- C_2$ and $Q =  2U\ln U + 2(1-U) \ln(1-U) + 1$, with $U$ uniformly distributed on $[0,1]$ and $E[Q] = 0$, in which case we can take any $p \in \mathbb{N}_+$ and $H_p = 1 - 2p E[ U^{p-1} (1-U)] = (p-1)/(p+1) < 1$ in Assumption~\ref{A.PhiAssumption} {\sc [linear0]}. Lemma~\ref{L.Moments} also gives that $E[|R^{(k)}|^p]$ is uniformly bounded in $k$ for all $p \geq 1$. 

For the PageRank algorithm studied in  \cite{Volk_Litv_08,Jel_Olv_10,Chen_Litv_Olv_14} we have $\{ C_i \}_{1 \leq i \leq N}$ i.i.d.~and independent of $N$, $|C_i| \leq c < 1$ a.s., and $E[ |C_1|^p ] \leq c^p/E[N]$ for any $p \geq 1$. Hence, we can take $p = 1$ and $H_1 = E[N] E[|C_1|] \leq c < 1$ in Assumption~\ref{A.PhiAssumption}. Furthermore, Theorem~\ref{T.LpConvergence}(ii) gives that $E[ |R^{(k)} - R|^q ] = O( \gamma^k)$ for some $0 < \gamma < 1$ provided $E[ |Q|^q + N^q] < \infty$, which in turn gives the uniform boundedness of $E[|R^{(k)}|^q]$. 

\item Using the inequality 
\begin{equation} \label{eq:MaxLipschitz}
\left| \max_{1\leq i \leq n} \{x_i\} - \max_{1 \leq i \leq n} \{ y_i \} \right| \leq \max_{1 \leq i \leq n} |x_i - y_i| \leq \sum_{i=1}^n |x_i - y_i|
\end{equation} 
for any real numbers $\{x_i, y_i\}$ and any $n \geq 1$, we obtain that the maximum SFPE \eqref{eq:Max} satisfies Assumption~\ref{A.Lipschitz} with $\varphi(t) = |t|$ as well. Furthermore, in the analysis of queueing networks with parallel servers and synchronization requirements from \cite{Kar_Kel_Suh_94, Olv_Ruiz_14}, where $T \equiv 0$ (equivalently, $Q \equiv 1$), the stability condition of the system implies that $H_p < 1$ for any $p \geq 1$ whenever the system is stable. Lemma~\ref{L.Moments} then implies that $E[|R^{(k)}|^p]$ is uniformly bounded in $k$ for all $p \geq 1$.

\item In the case of the discounted  tree sum SFPE \eqref{eq:DiscountedTrees}, inequality \eqref{eq:MaxLipschitz} implies that we can also take $\varphi(t) = |t|$ in Assumption~\ref{A.Lipschitz}.  For the analysis of the {\tt FIND} algorithm in \cite{Devroye_01} in particular, we have $N \equiv 2$, $C_1 = U = 1 - C_2$ and $Q \equiv 1$, with $U$ uniformly distributed on $[0,1]$, and we can take  $H_p = 2 E[U^p] =2/(p+1) < 1$ for any $p > 1$ in Assumption~\ref{A.PhiAssumption}. Lemma~\ref{L.Moments} then gives that $E[|R^{(k)}|^p]$ is uniformly bounded in $k$ for all $p > 1$

\item To see that \eqref{eq:IsingModel} also satisfies Assumption~\ref{A.Lipschitz} with $\varphi(t) = |t|$ (in this case $C_i \equiv \tanh(\beta)$ for all $i\geq 1$), let $c = \tanh(\beta) \in [0, 1)$ (since $\beta \geq 0$) and note that the function 
\begin{align*}
f(x) &= \arctanh (c \tanh(x)) = \frac{1}{2} \ln \left( \frac{1+c (e^{2x}-1)/(e^{2x}+1) }{1-c (e^{2x}-1)/(e^{2x}+1)} \right) \\
&= \frac{1}{2} \ln \left( \frac{e^{2x}(1+c) + 1-c }{e^{2x}(1-c)  + 1+c } \right)
\end{align*}
has derivative
$$f'(x)  = \frac{4c}{2(1 + c^2) + (e^{2x} + e^{-2x}) (1 - c^2)  } = \frac{2c}{1+c^2 + \cosh(2x) (1-c^2)},$$
and therefore satisfies 
$$|f(x) - f(y)| = | f'(\xi) | |x-y| \leq c | x - y| , \qquad \text{for some $\xi$ between $x$ and $y$}.$$
Assumption~\ref{A.PhiAssumption} is then satisfied for $p = 1$ and $H_1 \leq E[N] \tanh(\beta)$, with $H_p < 1$ at high temperatures ($\beta < 1/E[N]$). Moreover,  since $|f(x)| \leq c|x|$, $R^{(k)}$ in the  ``free entropy" SFPE \eqref{eq:IsingModel} is smaller or equal than $\tilde R^{(k)}$, where 
$$\tilde R^{(k)} = |Q| + \sum_{i=1}^N \tanh(\beta) \tilde R_i^{(k-1)}.$$
Hence, provided $\beta < 1/E[N]$, Theorem~\ref{T.LpConvergence}(ii) gives that for any $p \geq 1$ for which $E[|Q|^p + N^p] < \infty$, $E[|R^{(k)}|^p]$ is uniformly bounded in $k$.
\end{itemize}
\end{examples}

Our first result establishes the convergence in mean of $d_p(\hat F_{k,m}, F_k)$ under the ``optimal" moment conditions, that is, assuming only that  $\max_{0 \leq j \leq k} E[|R^{(j)}|^p] < \infty$. In view of Remark~\ref{R.Assumptions}(ii), this is implied in all our examples by $E\left[ \left( \sum_{i=1}^N \varphi(C_i) \right)^p \right] < \infty$.  This result was previously proven in  \cite{Chen_Olv_15} for the linear SFPE \eqref{eq:Linear} for $p = 1$.

\begin{thm}\label{T.MeanConvergence}
Fix $1 \leq p < \infty$ and suppose that $\Phi$ satisfies Assumption~\ref{A.PhiAssumption} ,or Assumption~\ref{A.PhiAssumption} {\sc [linear0]}, for $p$. Assume further that for any fixed $k \in \mathbb{N}$,  $\max_{0 \leq j \leq k} E[|R^{(j)}|^p] < \infty$. Let $\{ R^{(j)}_1, \dots, R^{(j)}_m\}$ be an i.i.d.~sample from distribution $F_j$, and let $F_{j,m}$ denote their corresponding empirical distribution function.   Then, 
\begin{equation*}
E\left[ d_p(\hat F_{k,m}, F_{k})^p \right]  \leq \left( \sum_{r=0}^k (H_p^{1/p})^{r} \right)^{p-1} \sum_{j=0}^k (H_p^{1/p})^{k-j} E\left[  d_p(F_{j,m}, F_{j})^p  \right],
\end{equation*}
where $0 < H_p < \infty$ is the same from Assumption \ref{A.PhiAssumption}. Moreover, if $\max_{0 \leq j \leq k} E[ |R^{(j)}|^q] < \infty$ for $q > p \geq 1$, $q \neq 2p$, then
\begin{align*}
E\left[ d_p(\hat F_{k,m}, F_{k})^p \right] & \leq  K \left( \sum_{r=0}^k (H_p^{1/p})^{r} \right)^{p-1}  \sum_{j=0}^k (H_p^{1/p})^{k-j} (E[|R^{(j)}|^q])^{p/q} \cdot m^{-\min\{(q-p)/q, \,1/2\}},
\end{align*}
where $K = K(p,q)$ is a constant that only depends on $p$ and $q$. 
\end{thm}

\begin{remarks}
\begin{enumerate} \renewcommand{\labelenumi}{(\roman{enumi})}
\item Note that Assumption \ref{A.PhiAssumption} does not require that $H_p < 1$, i.e., it is not necessary for $\Phi$ to define a contraction for the algorithm to work. However, when $H_p < 1$ the bound provided by Theorem~\ref{T.MeanConvergence} becomes independent of $k$, ensuring that the complexity of the population dynamics algorithm remains linear in $k$, rather than exponential, i.e., $(E[N])^k$, as the naive algorithm. When $H_p \geq 1$ for all $p \geq 1$ the bound given above may grow with the level of the recursion, i.e., the value of $k$, and the convergence of the sequence $\{\mu_k\}$ as $k \to \infty$ may not be guaranteed. 

\item Even in the case when $H_p \geq 1$ for all $p \geq 1$, the explicit bounds provided by Theorem~\ref{T.MeanConvergence} may be useful for determining whether endogenous solutions exist, since they guarantee that we can accurately approximate $R^{(k)}$. 

\item We also point out that the first inequality in Theorem~\ref{T.MeanConvergence} implies that the rate at which $E\left[ d_p(\hat F_{k,m}, F_{k})^p \right]$ converges to zero is determined by $\max_{0 \leq j \leq k} E[ d_p(F_{j,m}, F_j)]$. Since $d_p( F_{j,m}, F_j)$ corresponds to implementing the population dynamics algorithm by sampling without replacement from a ``perfect" i.i.d.~pool of observations from $\mu_{j-1}$, this convergence rate is in some sense optimal.

\item For all the examples given in Examples~\ref{E.Examples}, we have $H_p < 1$ and $\sup_{k \geq 0} E[|R^{(k)}|^q] < \infty$ for some $q > p$, making the bound provided by Theorem~\ref{T.MeanConvergence} independent of $k$. Moreover, for the Quicksort and {\tt FIND} algorithms, as well as for the queuing networks with parallel servers and synchronization requirements, the best possible rate of convergence is achieved, i.e., $E[d_p(\hat F_{k,m}, F_k)^p] = O( m^{-1/2})$ uniformly in $k$.
\end{enumerate}
\end{remarks}

We now turn our attention to the almost sure convergence of $d_p(\hat F_{k,m}, F_k)$, for which we provide two different results. The first one holds under Assumption~\ref{A.PhiAssumption} as above, but under rather strong moment conditions. Note that for the linear case Assumption~\ref{A.PhiAssumption}, in its general form, holds for any $p \geq 1$ for which $E\left[ \left( \sum_{i=1}^N |C_i| \right)^p \right] < \infty$ by Remark~\ref{R.Assumptions}(i).  Allowing Assumption~\ref{A.PhiAssumption} to hold for only $E[ X_i - Y_i] = 0$ is important for guaranteeing that we can choose $H_p < 1$ in Theorem~\ref{T.MeanConvergence}, but is unimportant for the almost sure convergence of the algorithm.

\begin{thm} \label{T.AlmostSure1}
Fix $1 \leq p < \infty$ and suppose that $\Phi$ satisfies Assumption~\ref{A.PhiAssumption} for both $p$ and $2p$. Assume further that for any fixed $k \in \mathbb{N}$, $\max _{0 \leq j \leq k} E[ (R^{(j)})^{2p} (\log |R^{(j)}|)^+] < \infty$. Then,
$$\lim_{m \to \infty} d_p(\hat F_{k,m}, F_k) = 0 \quad \text{a.s.}$$
\end{thm}

The moment condition requiring the finiteness of the $2p$ absolute moment also appears in some related (stronger) results for the convergence of the Wasserstein distance between a distribution function and its empirical measure, specifically, concentration inequalities \cite{Fou_Gui_15} and a central limit theorem \cite{Barr_Gin_Mat_99}. In our case, where we seek only to establish the almost sure convergence of the algorithm, this condition is too strong, so we provide below an improved result under the finer Assumption~\ref{A.Lipschitz}.

\begin{thm} \label{T.AlmostSure2}
Fix $1 \leq p < \infty$ and suppose that $\Phi$ satisfies Assumption~\ref{A.Lipschitz}. Assume further that $E[ |R^{(0)}|^{p+\delta} + Z^{p+\delta}] < \infty$ for some $\delta > 0$, where $Z = \sum_{i=1}^N \varphi(C_i)$. Then, for any fixed $k \in \mathbb{N}$, 
$$\lim_{m \to \infty} d_p(\hat F_{k,m}, F_k) = 0 \qquad \text{a.s.}$$
\end{thm}

Our last result relates the convergence of $d_p(\hat F_{k,m}, F_k)$ to the consistency of estimators based on the pool $\mathscr{P}^{(k,m)}$. More precisely, the value of the algorithm lies in the fact that it efficiently produces a sample of identically distributed random variables whose distribution is approximately $F_k$. A natural estimator for quantities of the form $E[ h(R^{(k)})]$ is then given by
\begin{equation} \label{eq:hEstimator}
\frac{1}{m} \sum_{i=1}^m h (\hat R_i^{(k,m)}) = \int_{\mathbb{R}} h(x) d \hat F_{k,m}(x).
\end{equation}
However, the random variables in $\mathscr{P}^{(k,m)}$ are not independent of each other, and the consistency of such estimators requires proof. In the sequel, the symbol $\stackrel{P}{\to}$ denotes convergence in probability.

\begin{defn}
We say that $\Theta_n$ is a {\em weakly} consistent estimator for $\theta$ if $\Theta_n \stackrel{P}{\to} \theta$ as $n \to \infty$. We say that it is a {\em strongly} consistent estimator for $\theta$ if $\Theta_n \to \theta$ a.s. 
\end{defn}

Our last result shows the consistency of estimators of the form in \eqref{eq:hEstimator} for a broad class of functions.

\begin{prop} \label{P.Consistency}
Fix $1 \leq p < \infty$ and suppose that $h: \mathbb{R} \to \mathbb{R}$ satisfies $|h(x)| \leq C (1 + |x|^p)$ for all $x \in \mathbb{R}$ and some constant $C > 0$. Then, the following hold:
\begin{enumerate}
\item If $E[ d_p(\hat F_{k,m}, F_k)^p ] \to 0$ as $m \to \infty$, then \eqref{eq:hEstimator} is a weakly consistent estimator for $E[ h(R^{(k)})]$ for each fixed $k \in \mathbb{N}$. 

\item If $d_p(\hat F_{k,m}, F_k) \to 0$ a.s.,~as $m \to \infty$, then \eqref{eq:hEstimator} is a strongly consistent estimator for $E[ h(R^{(k)})]$ for each fixed $k \in \mathbb{N}$.
\end{enumerate}
\end{prop}

We conclude that the population dynamics algorithm can be used to efficiently generate sample pools of random variables having a distribution that closely approximates that of the special endogenous solution to SFPEs of the form in \eqref{eq:SFPE}. Furthermore, these sample pools can be used to produce consistent estimators for a broad class of functions. The gain of efficiency of the algorithm compared to a naive Monte Carlo approach, combined with the consistency guarantees proved in this paper, make it extremely useful for the numerical analysis of many problems where SFPEs appear.

\section{Proofs} \label{S.Proofs}

This section includes the proofs of Theorems~\ref{T.MeanConvergence}, \ref{T.AlmostSure1}, \ref{T.AlmostSure2}, Proposition~\ref{P.Consistency}, Theorem~\ref{T.LpConvergence}, and of Lemma~\ref{L.Moments}, in that order. The last two appear at the end since they are not directly related to the Population Dynamics algorithm. The first four proofs are based on a construction of the pools $\{ \mathscr{P}^{(j,m)}: 0 \leq j \leq k\}$ where we carefully couple the random variables $\{ \hat R_i^{(j,m)} \}$ with i.i.d.~observations from their limiting distribution $F_j$. 


To start, for any $k \in \mathbb{N}$ let 
\begin{equation} \label{eq:AllRandomVectors}
\mathscr{E}_k = \left\{ \left(Q_i^{(j)}, N_i^{(j)}, \{ C_{(i,r)}^{(j)}\}_{r \geq 1}, \{ U_{(i,r)}^{(j)} \}_{r \geq 1} \right): i \geq 1, \, 0 \leq j \leq k \right\}
\end{equation}
be a collection of i.i.d.~random vectors where $\left(Q_i^{(j)}, N_i^{(j)}, \{ C_{(i,r)}^{(j)}\}_{r \geq 1} \right)$ has the same distribution as the generic branching vector $(Q, N, \{C_r\}_{i \geq 1})$ and the $\{U_{(i,r)}^{(j)}\}_{r \geq 1}$ are i.i.d.~random variables uniformly distributed in $[0,1]$, independent of $\left(Q_i^{(j)}, N_i^{(j)}, \{ C_{(i,r)}^{(j)}\}_{r \geq 1} \right)$. Next, we recursively construct a sequence of random variables $\{ (\hat R_i^{(j,m)}, R_i^{(j)}): 1 \leq i \leq m, \, 0 \leq j \leq k\}$ as follows:
\begin{enumerate} \renewcommand{\labelenumi}{\roman{enumi}.} 
\item Set $\hat R_i^{(0)} = F_0^{-1}(U_{(i,1)}^{(0)}) = R_i^{(0,m)}$, for $1 \leq i \leq m$; define 
$$\hat F_{0,m}(x) = \frac{1}{m} \sum_{i=1}^m 1(\hat R_i^{(0,m)} \leq x) = F_{0,m}(x).$$
\item For $1 \leq j \leq k$ and each $1 \leq i \leq m$,
 \begin{align*}
        \hat R_{i}^{(j,m)} &= \Phi\left( Q_i^{(j)}, N_i^{(j)}, \{ C_{(i,r)}^{(j)} \}_{r \geq 1}, \{ \hat F^{-1}_{j-1,m} (U_{(i,r)}^{(j)}) \}_{r \geq 1} \right) \qquad \text{and} \\
        R_{i}^{(j)} &= \Phi\left( Q_i^{(j)}, N_i^{(j)}, \{ C_{(i,r)}^{(j)} \}_{r \geq 1}, \{ F^{-1}_{j-1} (U_{(i,r)}^{(j)}) \}_{r \geq 1} \right);
    \end{align*}
define
$$\hat F_{j,m}(x) = \frac{1}{m} \sum_{i=1}^m 1(\hat R_i^{(j,m)} \leq x) \quad \text{and} \quad F_{j,m}(x) = \frac{1}{m} \sum_{i=1}^m 1(R_i^{(j)} \leq x).$$    
\end{enumerate}
Note that the random variables $\{ R_i^{(j)}\}_{i=1}^m$ are i.i.d.~and have distribution $F_j$, and therefore, $F_{j,m}$ is an empirical distribution function for $F_j$. The distribution functions $\hat F_{j,m}$ are those obtained through the population dynamics algorithm. 

Throughout the proofs we will also use repeatedly the sigma-algebra $\mathcal{F}_k = \sigma (\mathscr{E}_k)$ for $k \in \mathbb{N}$. We point out that all the random variables $\{ ( \hat R_i^{(k,m)}, R_i^{(k)}) : i \geq 1\}$ are measurable with respect to $\mathcal{F}_k$ for all $m \geq 1$.


We are now ready to prove Theorem~\ref{T.MeanConvergence}. 

\begin{proof}[Proof of Theorem~\ref{T.MeanConvergence}]
Let $\{ (\hat R_i^{(j,m)}, R_i^{(j)}): 1 \leq i \leq m, \, 0 \leq j \leq k\}$ be a sequence of random vectors constructed as explained above. 

Next, note that from the triangle inequality we obtain
\begin{equation} \label{eq:triangle}
d_p(\hat F_{j,m}, F_{j}) \leq d_p(\hat F_{j,m}, F_{j,m}) + d_p(F_{j,m}, F_{j}).
\end{equation}

Now let $\chi$ be a Uniform(0,1) random variable independent of everything else, and define the random variables 
\begin{align*}
\hat R^{(j,m)} &= \sum_{i=1}^m \hat R_i^{(j,m)} 1((i-1)/m < \chi \leq i/m) \qquad \text{and}  \\
\qquad R^{(j)} &= \sum_{i=1}^m R_i^{(j)} 1((i-1)/m < \chi \leq i/m),
\end{align*}
which conditionally on $\mathcal{F}_j$ are distributed according to $\hat F_{j,m}$ and $F_{j,m}$, respectively. Then, from the definition of $d_p$ we have
\begin{align}
d_p(\hat F_{j,m}, F_{j,m})^p &= \inf_{X\sim \hat F_{j,m}, Y\sim F_{j,m}} E\left[ \left. |X - Y|^p \right| \mathcal{F}_j \right] \notag \\
&\leq E\left[ \left. \left| \hat R^{(j,m)} - R^{(j)} \right|^p \right| \mathcal{F}_j \right] \notag \\
&= \frac{1}{m} \sum_{i=1}^m \left| \hat R_i^{(j,m)} - R_i^{(j)} \right|^p. \label{eq:ExplicitBound}
\end{align}
It follows from the observation that the random variables $X_i^{(j)} = \hat R_i^{(j,m)} - R_i^{(j)}$ are identically distributed, that
$$E\left[ d_p(\hat F_{j,m}, F_{j,m})^p \right] \leq E\left[ \left| \hat R_1^{(j,m)} - R_1^{(j)} \right|^p \right] .$$  

Next, suppose first that Assumption~\ref{A.PhiAssumption} for any $\{ X_i \}$ and $\{Y_i\}$, and note that
\begin{align*}
E\left[ \left| \hat R_1^{(j,m)} - R_1^{(j)} \right|^p \right] &= E\left[\left| \Phi\left( Q_1^{(j)}, N_1^{(j)}, \{ C_{(1,r)}^{(j)} \}_{r \geq 1}, \{ \hat F^{-1}_{j-1,m} (U_{(1,r)}^{(j)}) \}_{r \geq 1} \right) \right. \right. \\
&\hspace{10mm} \left. \left.  - \Phi\left( Q_1^{(j)}, N_1^{(j)}, \{ C_{(1,r)}^{(j)} \}_{r \geq 1}, \{ F^{-1}_{j-1} (U_{(1,r)}^{(j)}) \}_{r \geq 1} \right) \right|^p  \right] \\
&\leq  H_p E\left[ \left|\hat F^{-1}_{j-1,m} (U_{(1,1)}^{(j)}) - F^{-1}_{j-1} (U_{(1,1)}^{(j)}) \right|^p \right] \\
&= H_p  E\left[ d_p(\hat F_{j-1,m}, F_{j-1})^p \right] .
\end{align*}
For the linear case when only Assumption~\ref{A.PhiAssumption} {\sc [linear0]} holds, 
note that 
\begin{align*}
E[ \hat F_{j-1,m}^{-1}(U) - F_{j-1}^{-1}(U) ] &= E\left[ \sum_{i=1}^N C_i \right] E[  \hat R_1^{(j-2,m)} - R^{(j-2)} ]  \\
&= \left( E\left[ \sum_{i=1}^N C_i \right] \right)^{j-1} E[ \hat R_1^{(0,m)} - R^{(0)} ]  = 0, 
\end{align*}
and therefore, 
\begin{align*}
E\left[ \left| \hat R_1^{(j,m)} - R_1^{(j)} \right|^p \right] &= E\left[\left| \sum_{r=1}^{N_1^{(j)}} C_{(1,r)}^{(j)} \left( \hat F^{-1}_{j-1,m} (U_{(1,r)}^{(j)}) - F^{-1}_{j-1} (U_{(1,r)}^{(j)}) \right)  \right|^p  \right] \\
&\leq H_p E\left[ \left|\hat F^{-1}_{j-1,m} (U_{(1,1)}^{(j)}) - F^{-1}_{j-1} (U_{(1,1)}^{(j)}) \right|^p \right] \\
&= H_p  E\left[ d_p(\hat F_{j-1,m}, F_{j-1})^p \right] .
\end{align*}

It now follows from  \eqref{eq:triangle} and Minkowski's inequality,  that
\begin{align*}
\left( E\left[ d_p(\hat F_{j,m}, F_{j})^p \right] \right)^{1/p} &\leq \left( E\left[ \left( d_p(\hat F_{j,m}, F_{j,m}) + d_p(F_{j,m}, F_{j}) \right)^p \right] \right)^{1/p} \\
&\leq \left( E\left[ d_p(\hat F_{j,m}, F_{j,m})^p  \right]  \right)^{1/p} +  \left( E\left[  d_p(F_{j,m}, F_{j})^p  \right]  \right)^{1/p} \\
&\leq  \left( H_p E\left[ d_p(\hat F_{j-1,m}, F_{j-1})^p \right] \right)^{1/p} + \left( E\left[  d_p(F_{j,m}, F_{j})^p  \right]  \right)^{1/p}.
\end{align*}
 Iterating the recursion above we obtain
\begin{align*}
\left( E\left[ d_p(\hat F_{j,m}, F_{j})^p \right] \right)^{1/p} &\leq \sum_{r=1}^{j} (H_p^{1/p})^{j-r} \left( E\left[  d_p(F_{r,m}, F_{r})^p  \right]  \right)^{1/p}  + (H_p^{1/p})^j \left(E\left[ d_p(\hat F_{0,m}, F_{0})^p \right] \right)^{1/p} \\
&= \sum_{r=0}^{j} (H_p^{1/p})^{j-r} \left( E\left[  d_p(F_{r,m}, F_{r})^p  \right]  \right)^{1/p}. 
\end{align*}
Now let $\lambda_{j,r} = (H_p^{1/p})^{j-r} \left( \sum_{r=0}^j (H_p^{1/p})^{j-r} \right)^{-1}$ and use the fact that $g(x) = x^{1/p}$ is concave to obtain
\begin{align*}
\left( \sum_{r=0}^j (H_p^{1/p})^{j-r} \right)^{-1} \left( E\left[ d_p(\hat F_{j,m}, F_{j})^p \right] \right)^{1/p} &\leq \sum_{r=0}^j \lambda_{j,r} \left( E\left[  d_p(F_{r,m}, F_{r})^p  \right]  \right)^{1/p} \\
&\leq \left( \sum_{r=0}^j \lambda_{j,r} E\left[  d_p(F_{r,m}, F_{r})^p  \right] \right)^{1/p},
\end{align*}
or equivalently,
$$E\left[ d_p(\hat F_{j,m}, F_{j})^p \right]  \leq \left( \sum_{s=0}^j (H_p^{1/p})^{s} \right)^{p-1} \sum_{r=0}^j (H_p^{1/p})^{j-r} E\left[  d_p(F_{r,m}, F_{r})^p  \right].$$
This completes the first part of the proof.

Next, assume that $\max_{0 \leq r \leq k} E[ |R^{(r)}|^q] < \infty$ for $q > p \geq 1$, $q \neq 2p$, and use Theorem~1 in \cite{Fou_Gui_15} to obtain that
$$E\left[ d_p(F_{r,m}, F_r)^p \right] \leq C (E[|R^{(r)}|^q])^{p/q} \left( m^{-1/2} + m^{-(q-p)/q} \right),$$
where $C = C(p,q)$ is a constant that does not depend on $F_r$. The second statement of the theorem now follows. 
\end{proof}

\bigskip

We now turn to the proof of Theorem~\ref{T.AlmostSure1}. To simplify its exposition we first provide a preliminary result for the mean Wasserstein distance between a distribution and its empirical distribution function. 

\begin{lemma} \label{L.MeanSeries}
Let $G$ be a distribution on $\mathbb{R}$ and let $\{ X_i\}_{i \geq 1}$ be i.i.d.~random variables distributed according to $G$. Suppose $E[ |X_1|^{q} (\log |X_1|)^+] < \infty$ for some $q \geq 2$,  and let $G_m(x) = m^{-1} \sum_{i=1}^m 1(X_i \leq x)$ denote the empirical distribution function of the $\{X_i\}$. Then,
$$\sum_{m = 1}^\infty \frac{1}{m} E[ d_{q}(G_m, G)^{q}] < \infty.$$ 
\end{lemma}

\begin{proof}
Fix $\epsilon > 0$ and define for $x \geq 0$ the functions
$$a(x) = \min\{ 1/\overline{G}(x), \, x^{q+\epsilon}\} \qquad \text{and} \qquad b(x) = \min\{ 1/G(-x), x^{q+\epsilon}\}.$$
Next, use Proposition~7.14 in \cite{Bob_Led_17} followed by the monotonicity of the $L_p$ norm, to see that
\begin{align}
\sum_{m=1}^\infty \frac{1}{m} E\left[ d_{q}(G_{m}, G)^{q} \right] &\leq q 2^{q-1} \sum_{m=1}^\infty \frac{1}{m} \int_{-\infty}^\infty |x|^{q-1} E\left[ \left| G_{m}(x) - G(x) \right| \right] dx \notag \\
&\leq q 2^{q-1} \sum_{m=1}^\infty \frac{1}{m} \int_{-b^{-1}(m)}^{a^{-1}(m)} |x|^{q-1} \left( E\left[ \left( G_{m}(x) - G(x) \right)^2 \right] \right)^{1/2} dx \notag \\
&\hspace{5mm} + q 2^{q-1} \sum_{m=1}^\infty \frac{1}{m} \int_{a^{-1}(m)}^{\infty} x^{q-1} E\left[  \overline{G}_{m}(x) + \overline{G}(x) \right] dx \notag \\
&\hspace{5mm} + q 2^{q-1} \sum_{m=1}^\infty \frac{1}{m} \int_{-\infty}^{-b^{-1}(m)} |x|^{q-1} E\left[ G_{m}(x) + G(x)  \right] dx \notag \\
&= q 2^{q-1} \sum_{m=1}^\infty \frac{1}{m} \int_{-b^{-1}(m)}^{a^{-1}(m)} |x|^{q-1} \sqrt{ \frac{G(x) \overline{G}(x)}{m}} dx \label{eq:Middle} \\
&\hspace{5mm} + q 2^{q} \sum_{m=1}^\infty \frac{1}{m} \int_{a^{-1}(m)}^{\infty} x^{q-1}  \overline{G}(x)  \, dx \label{eq:RightTail} \\
&\hspace{5mm} + q 2^{q} \sum_{m=1}^\infty \frac{1}{m} \int_{-\infty}^{-b^{-1}(m)} |x|^{q-1} G(x) \, dx, \label{eq:LeftTail}
\end{align}
where $g^{-1}(t) = \inf\{ x \in \mathbb{R}: g(x) \geq t\}$ is the generalized inverse of function $g$. 

Next, to bound \eqref{eq:Middle} note that 
\begin{align*}
&\sum_{m=1}^\infty \frac{1}{m^{3/2}} \int_{-b^{-1}(m)}^{a^{-1}(m)} |x|^{q-1} \sqrt{ G(x) \overline{G}(x)} \, dx \\
&\leq  \sum_{m=1}^\infty \frac{1}{m^{3/2}} \int_{0}^{a^{-1}(m)} x^{q-1} \sqrt{ \overline{G}(x)} \, dx + \sum_{m=1}^\infty \frac{1}{m^{3/2}} \int_{-b^{-1}(m)}^0 (-x)^{2p-1} \sqrt{ G(x)} \, dx \\
&=   \int_{0}^{\infty} \sum_{m=\lfloor a(x) \rfloor +1}^\infty\frac{x^{q-1}}{m^{3/2}}  \sqrt{ \overline{G}(x)} \, dx +   \int_0^{\infty} \sum_{m=\lfloor b(x) \rfloor+1}^\infty \frac{x^{q-1}}{m^{3/2}}  \sqrt{ G(-x)} \, dx,
\end{align*}
where in the last equality we used the observation that $\{ x < a^{-1}(m)\} = \{a(x) < m\}$, respectively, $\{ x < b^{-1}(m)\} = \{b(x) < m\}$. Now note that for any $n \geq 0$ we have
\begin{align*}
\sum_{m=n+1}^\infty \frac{1}{m^{3/2}} &\leq \sum_{m=n+1}^\infty \left( \frac{m+1}{m} \right)^{3/2} \int_{m}^{m+1} \frac{1}{t^{3/2}} \, dt \\
&\leq \left( 1 + \frac{1}{n+1} \right)^{3/2} \int_{n+1}^\infty t^{-3/2} \, dt \leq 2^{5/2}  (n+1)^{-1/2}  .
\end{align*}
Hence, \eqref{eq:Middle} is bounded from above by a constant times
\begin{align*}
&\int_0^\infty x^{q-1} \sqrt{ \overline{G}(x)}  (\lfloor a(x) \rfloor+1)^{-1/2} \, dx + \int_0^\infty x^{q-1} \sqrt{ G(-x)}  (\lfloor b(x) \rfloor+1)^{-1/2} \, dx  \\
&\leq 2 + \int_1^\infty x^{q-1} \sqrt{ \frac{ \overline{G}(x)}{a(x)} }  \, dx + \int_1^\infty x^{q-1} \sqrt{ \frac{G(-x)}{b(x)} } \, dx  \\
&= 2+ \int_{\{x \geq 1: 1/\overline{G}(x) \leq x^{q+\epsilon}\}}  x^{q-1} \overline{G}(x)  \, dx + \int_{\{x \geq 1: 1/\overline{G}(x) > x^{q+\epsilon}\}} x^{q/2-1-\epsilon/2} \sqrt{ \overline{G}(x) } \, dx  \\
&\hspace{5mm} + \int_{\{x \geq 1: 1/G(-x) \leq x^{q+\epsilon}\}}  x^{q-1} G(-x)  \, dx \\
&\hspace{5mm} + \int_{\{x \geq 1: 1/G(-x) > x^{q+\epsilon}\}} x^{q/2-1-\epsilon/2} \sqrt{ G(-x)} \, dx  \\
&\leq 2+  \int_1^\infty x^{q-1} \overline{G}(x) \, dx + \int_1^\infty x^{q-1} G(-x) \, dx + 2 \int_1^\infty x^{-1-\epsilon} \, dx \\
&\leq 2 +   \frac{1}{q} \int_1^\infty x^{q} G(dx)  + \frac{1}{q} \int_{-\infty}^{-1} (-x)^{q} G(dx)  + \frac{2}{\epsilon} \\
&\leq 2 + \frac{1}{q} E[ |X_1|^q ] + \frac{2}{\epsilon} < \infty. 
\end{align*}

To analyze \eqref{eq:RightTail} use the observation that $\{ x \geq a^{-1}(m)\} = \{ a(x) \geq m\}$ to obtain that
\begin{align*}
\sum_{m=1}^\infty \frac{1}{m} \int_{a^{-1}(m)}^{\infty} x^{q-1}  \overline{G}(x)  \, dx &= \int_{a^{-1}(1)}^\infty \sum_{m=1}^{\lfloor a(x) \rfloor} \frac{1}{m} \,  x^{q-1}  \overline{G}(x)  \, dx \\
&\leq \int_{a^{-1}(1)}^\infty x^{q-1}  \overline{G}(x) \sum_{m=1}^{\lfloor a(x) \rfloor} \frac{m+1}{m} \int_{m}^{m+1} \frac{1}{t} \, dt     \, dx \\
&\leq 2 \int_{a^{-1}(1)}^\infty x^{q-1}  \overline{G}(x) \int_1^{\lfloor a(x) \rfloor +1} t^{-1} \, dt \, dx \\
&\leq 2 \int_{a^{-1}(1)}^\infty x^{q-1} \overline{G}(x) \log( x^{q+\epsilon}+1) \, dx \\
&\leq 2 \log 2 + 2 (q+\epsilon) \sup_{t \geq 1} \frac{\log (t+1)}{\log t} \int_1^\infty x^{q-1} (\log x) \overline{G}(x) \, dx. 
\end{align*}
Since $\sup_{t \geq 1} \log(t+1)/\log t < \infty$ and 
\begin{align*}
\int_1^\infty x^{q-1} (\log x) \overline{G}(x) \, dx &= \left.  \frac{x^{q} ( \log x - 1)}{q} \overline{G}(x) \right|_1^\infty + \int_1^\infty  \frac{x^{q} (\log x - 1)}{q} G(dx) \\
&= \frac{\overline{G}(1)}{q} + \frac{E[ |X_1|^{q} (\log X_1 - 1) 1(X_1 \geq 1) ]}{q} \\
&\leq \frac{E[ |X_1|^q \log X_1 1(X_1 \geq 1)]  }{q} < \infty,
\end{align*}
we obtain that \eqref{eq:RightTail} is finite. Finally, the same steps used to bound \eqref{eq:RightTail} give that \eqref{eq:LeftTail} is bounded by 
$$q2^{q} \left(  2 \log 2 + 2(q+\epsilon) \frac{E[ |X_1|^{q} \log |X_1| 1(X_1 \leq -1)]}{q}  \sup_{t \geq 1} \frac{\log(t+1)}{\log t} \right) < \infty.$$
\end{proof}

\bigskip

We now give the proof for the first result on the almost sure convergence of the algorithm. The idea of the proof is to first identify a recursive formula for the Wasserstein distance $d_p(\hat F_{k,m}, F_k)$ as it was done for the convergence in mean theorem. Once we do this, the main difficulty lies in ensuring that the errors in the bound converge sufficiently fast to satisfy the criterion for almost sure convergence in the Borel-Cantelli lemma. In the case when we have a bit more than $2p$ finite moments this can be done using Chebyshev's inequality, similarly to the proof of the strong law of large numbers under finite fourth moment conditions. We start with this case below.

\begin{proof}[Proof of Theorem~\ref{T.AlmostSure1}]
We will start the proof by deriving an upper bound for $d_p(\hat F_{k,m}, F_k)$. To this end, we construct the random variables $\{ (\hat R_i^{(j,m)}, R_i^{(j)}) : 1 \leq i \leq m, \, 0 \leq j \leq k\}$ according to the construction given at the beginning of the section. Recall that $\mathcal{F}_j = \sigma(\mathscr{E}_j)$, where $\mathscr{E}_j$ is given by \eqref{eq:AllRandomVectors}, and that Assumption~\ref{A.PhiAssumption} holds for both $p$ and $2p$. 

We start by noting that the triangle inequality followed by \eqref{eq:ExplicitBound} give
\begin{align*}
d_p(\hat F_{k,m}, F_k) &\leq d_p(\hat F_{k,m}, F_{k,m}) + d_p(F_{k,m}, F_k) \\
&\leq \left( \frac{1}{m} \sum_{i=1}^m \left| \hat R_i^{(k,m)} - R_i^{(k)} \right|^p \right)^{1/p} + d_p(F_{k,m}, F_k) .
\end{align*}

Next, define for $j \geq 1$, $X_i^{(j,m)} =  \left| \hat R_i^{(j,m)} - R_i^{(j)} \right|^p$ and note that by construction, the random variables $\{ X_i^{(j,m)} \}_{i \geq 1}$ are identically distributed and conditionally independent given $\mathcal{F}_{j-1}$. Now set $Z_i^{(j,m)} = X_i^{(j,m)}  - E[ X_1^{(j,m)}  | \mathcal{F}_{j-1} ]$ and note that 
\begin{align}
E[ X_1^{(j,m)} | \mathcal{F}_{j-1} ] &= E\left[ \left| \Phi\left( Q_1^{(j)}, N_1^{(j)}, \{ C_{(1,r)}^{(j)} \}_{r \geq 1}, \{ \hat F^{-1}_{j-1,m} (U_{(1,r)}^{(j)}) \}_{r \geq 1} \right) \right. \right. \notag \\
&\hspace{10mm} \left. \left. \left.  - \Phi\left( Q_1^{(j)}, N_1^{(j)}, \{ C_{(1,r)}^{(j)} \}_{r \geq 1}, \{ F^{-1}_{j-1} (U_{(1,r)}^{(j)}) \}_{r \geq 1} \right) \right|^p \right| \mathcal{F}_{j-1}    \right] \notag \\
&\leq H_p E\left[ \left.  \left|  \hat F^{-1}_{j-1,m} (U_{(1,1)}^{(j)}) - F^{-1}_{j-1} (U_{(1,1)}^{(j)}) \right|^p  \right| \mathcal{F}_{j-1}  \right]  \notag \\
&= H_p d_p(\hat F_{j-1,m}, F_{j-1})^p . \label{eq:CondMeanX}
\end{align}
It follows that
$$\frac{1}{m} \sum_{i=1}^m \left| \hat R_i^{(k,m)} - R_i^{(k)} \right|^p \leq \frac{1}{m} \sum_{i=1}^m Z_i^{(k,m)}   + H_p d_p(\hat F_{k-1,m}, F_{k-1})^p,$$
which in turn implies that
\begin{align}
d_p(\hat F_{k,m}, F_k) &\leq d_p(F_{k,m}, F_k) + \left(  \frac{1}{m} \sum_{i=1}^m Z_i^{(k,m)}   + H_p d_p(\hat F_{k-1,m}, F_{k-1})^p \right)^{1/p} \notag \\
&\leq d_p(F_{k,m}, F_k) + \left|  \frac{1}{m} \sum_{i=1}^m Z_i^{(k,m)} \right|^{1/p} + H_p^{1/p} d_p(\hat F_{k-1,m}, F_{k-1}), \label{eq:GenRec}
\end{align}
where in the last step we used the inequality $\left( x+y \right)^\beta \leq x^\beta + y^\beta$ for $0 < \beta \leq 1$ and $x,y \geq 0$.
Iterating \eqref{eq:GenRec} $k-1$ more times we obtain
\begin{align*}
d_p(\hat F_{k,m}, F_{k}) &\leq  \sum_{j=1}^k \left( d_p(F_{j,m}, F_j) + \left| \frac{1}{m} \sum_{i=1}^m Z_i^{(j,m)} \right|^{1/p} \right) (H_p^{1/p})^{k-j} \\
&\hspace{5mm} + (H_p^{1/p})^k d_p(\hat F_{0,m}, F_{0}) \\
&= \sum_{j=0}^k (H_p^{1/p})^{k-j} d_p(F_{j,m}, F_j) + \sum_{j=1}^k (H_p^{1/p})^{k-j}  \left| \frac{1}{m} \sum_{i=1}^m Z_i^{(j,m)} \right|^{1/p} . 
\end{align*}

Now note that by the Glivenko-Cantelli lemma and the strong law of large numbers, 
\begin{align*} 
\sup_{x \in \mathbb{R}} \left| F_{j,m}(x) - F_j(x) \right| &\to 0 \qquad \text{a.s.} \qquad \text{ and } \\
 \frac{1}{m} \sum_{i=1}^m |R^{(j)}_i |^p = \int_{-\infty}^\infty |x|^p dF_{j,m}(x) &\to \int_{-\infty}^\infty |x|^p dF_j(x) \quad \text{a.s.},
 \end{align*}
 as $m \to \infty$, and therefore, by Definition~6.8 and Theorem~6.9 in \cite{Villani_2009}, $d_p(F_{j,m}, F_j) \to 0$ a.s.~for each $j \geq 1$. It suffices then to show that for each $1 \leq j \leq k$ the sums $m^{-1}  \sum_{i=1}^m  Z_i^{(j,m)}  \to 0$ a.s.~as well. 
 
 To see this note that for any $\epsilon > 0$,
\begin{align*}
\sum_{m=1}^\infty P\left( \frac{1}{m} \sum_{i=1}^m Z_i^{(j,m)} > \epsilon \right) &\leq \sum_{m=1}^\infty \frac{1}{\epsilon^2 m^2} E\left[ \left( \sum_{i=1}^m Z_i^{(j,m)} \right)^2 \right] \\
&= \frac{1}{\epsilon^2} \sum_{m=1}^\infty \frac{1}{m} \left( E\left[ \left( Z_1^{(j,m)} \right)^2 \right] + (m-1) E\left[  Z_1^{(j,m)}  Z_2^{(j,m)}  \right] \right) \\
&= \frac{1}{\epsilon^2} \sum_{m=1}^\infty \frac{1}{m} E\left[ \var( X_1^{(j,m)}  | \mathcal{F}_{j-1})  \right] .
\end{align*}

Moreover, using the same arguments we used in the proof of Theorem~\ref{T.MeanConvergence}, we obtain that
\begin{align*}
\var( X_1^{(j,m)} | \mathcal{F}_{j-1}) &\leq E\left[ \left. ( X_1^{(j,m)})^2 \right| \mathcal{F}_{j-1} \right] \\
&= E\left[ \left( \Phi\left( Q_1^{(j)}, N_1^{(j)}, \{ C_{(1,r)}^{(j)} \}_{r \geq 1}, \{ \hat F^{-1}_{j-1,m} (U_{(1,r)}^{(j)}) \}_{r \geq 1} \right) \right. \right. \\
&\hspace{25mm} \left. \left. \left.  - \Phi\left( Q_1^{(j)}, N_1^{(j)}, \{ C_{(1,r)}^{(j)} \}_{r \geq 1}, \{ F^{-1}_{j-1} (U_{(1,r)}^{(j)}) \}_{r \geq 1} \right) \right)^{2p} \right| \mathcal{F}_{j-1}  \right] \\
&\leq H_{2p} \, E\left[ \left.  \left(  \hat F^{-1}_{j-1,m} (U_{(1,1)}^{(j)}) - F^{-1}_{j-1} (U_{(1,1)}^{(j)}) \right)^{2p}  \right| \mathcal{F}_{j-1}  \right] \quad \text{(by Assumption \ref{A.PhiAssumption})} \\
&= H_{2p} \, d_{2p}( \hat F_{j-1,m}, F_{j-1})^{2p} . 
\end{align*}

Next, note that by Theorem~\ref{T.MeanConvergence} we have
\begin{align*}
E\left[ d_{2p}( \hat F_{j-1,m}, F_{j-1})^{2p} \right] &\leq \left( \sum_{s=0}^{j-1} H_{2p}^s \right)^{2p-1} \sum_{r=0}^{j-1} H_{2p}^{j-1-r} E\left[ d_{2p}(F_{r,m}, F_{r})^{2p} \right].
\end{align*}
It follows that for any $1 \leq j \leq k$, 
\begin{align*}
&\sum_{m=1}^\infty P\left( \frac{1}{m} \sum_{i=1}^m Z_i^{(j,m)} > \epsilon \right) \\
&\leq  \frac{H_{2p}}{\epsilon^2} \sum_{m=1}^\infty \frac{1}{m}   E\left[ d_{2p}( \hat F_{j-1,m}, F_{j-1})^{2p} \right] \\
&\leq \frac{H_{2p}}{\epsilon^2}    \left( \sum_{s=0}^{j-1} H_{2p}^s \right)^{2p-1}  \sum_{r=0}^{j-1} H_{2p}^{j-1-r} \sum_{m=1}^\infty \frac{1}{m} E\left[ d_{2p}(F_{r,m}, F_{r})^{2p} \right].
\end{align*}

Finally, since by Lemma~\ref{L.MeanSeries} we have that 
$$\sum_{m=1}^\infty \frac{1}{m} E\left[ d_{2p}(F_{r,m}, F_{r})^{2p} \right] < \infty$$
for each $0 \leq r \leq j-1$, the Borel-Cantelli Lemma gives that $\lim_{m \to \infty} m^{-1}  \sum_{i=1}^m  Z_i^{(j,m)} = 0$ a.s. This completes the proof.
\end{proof}

\bigskip

We now move on to the proof of Theorem~\ref{T.AlmostSure2}, where we only have a bit more than $p$ finite moments. In this case, we cannot use Chebyshev's inequality to verify the condition for the Borel-Cantelli lemma, and a finer analysis of the errors is required. In particular, our proof uses the Lipschitz condition from Assumption~\ref{A.Lipschitz} to derive a large-deviations bound for the sum of independent random variables appearing in the recursive analysis of $d_p(\hat F_{k,m}, F_k)$. 
Before proceeding to the main proof, we give three preliminary results. The first one provides an upper bound for the generalized inverse of any distribution function having finite $q$ absolute moments. 

\begin{lemma} \label{L.InverseBound}
Let $G$ be a distribution function on $\mathbb{R}$, and let $G^{-1}$ be its generalized inverse. Suppose that $G$ has finite absolute moments of order $q > 0$. Then, for any $u \in (0,1)$, 
$$|G^{-1}(u)| \leq || X^+ ||_q (1-u)^{-1/q} + ||X^-||_q u^{-1/q} .$$
\end{lemma}

\begin{proof}
Let $X$ be a random variable having distribution $G$, and define $G_+(x) = P(X^+ \leq x) = G(x) 1(x \geq 0)$ and $G_-(x) = P(X^- \leq x) = P(X \geq -x) 1(x \geq 0) $. Then, 
$$G_+^{-1}(u) = \inf\{ x \in \mathbb{R}: G_+(x) \geq u \} = \inf\{ x\geq 0: G(x) \geq u \} = G^{-1}(u)^+,$$ 
while if we define $G_-^{*}$ to be the right-continuous generalized inverse of $G_-$, then
\begin{align*}
G_-^{*}(1-u) &= \inf\{ x \in \mathbb{R}: G_-(x) > 1- u \} \\
&= \inf\{ x \geq 0: 1 - G(-x) + P(X = -x) > 1-u \} \\
&= \inf\{ x \geq 0: G(-x) - P(X = -x) < u \} \\
&= -\inf \{ x \leq 0: G(x) \geq u \} = G^{-1}(u)^-.
\end{align*}
Now use Markov's inequality to obtain that for all $x > 0$, 
$$1 - G_+(x) \leq \min\{ 1, E[ (X^+)^q] \} x^{-q} \triangleq 1 - H_+(x)$$
and 
$$1 - G_-(x) \leq \min\{ 1, E[ (X^-)^q ] \} x^{-q} \triangleq 1 - H_-(x). $$
The first inequality implies that for any $u \in (0,1)$, 
\begin{align*}
G_+^{-1}(u) &= \inf \{ x \in \mathbb{R}: G_+(x) \geq u \} \\
&\leq \inf \{ x \in \mathbb{R}: H_+(x) \geq u \} = H_+^{-1}(u) = || X^+ ||_q (1-u)^{-1/q},
\end{align*}
while the second one plus the continuity of $H_-$ gives
\begin{align*}
G^{-1}(u)^- &= G_-^*(1-u) = \inf\{ x \in \mathbb{R}: G_-(x) > 1- u \}  \\
&\leq \inf\{ x \in \mathbb{R}: H_-(x) > 1- u \} \\
&= \inf\{ x \in \mathbb{R}: H_-(x) \geq 1- u \} =  H_-^{-1}(1-u) = ||X^-||_q u^{-1/q}.
\end{align*}
It follows that
$$|G^{-1}(u)| = G^{-1}(u)^+ + G^{-1}(u)^-  \leq || X^+ ||_q (1-u)^{-1/q} + ||X^-||_q u^{-1/q} .$$
\end{proof}

\bigskip

The next two preliminary results provide key steps for the proof of Theorem~\ref{T.AlmostSure2}, which essentially consist on giving a large-deviations bound (uniform in $m$) for the sample mean of (conditionally) i.i.d.~random variables. The random variables $\{ Y_i^{(j,m)}\}$ defined below will be used as upper bounds for $d_{p+\delta_{j+1}}(\hat F_{j,m}, F_j)$ in the proof of Theorem~\ref{T.AlmostSure2}, and the estimates we need have to be very tight considering that we no longer have finite second moments, so the rate of convergence to their mean can be very slow. The lemma below gives an upper bound for the truncated summands.

\begin{lemma} \label{L.ConditionalSums}
Fix $1 \leq p < \infty$ and $\epsilon > 0$. Suppose Assumption~\ref{A.Lipschitz} holds and $E[ |R^{(0)}|^{p+\delta} + Z^{p+\delta}] < \infty$ for some $\delta > 0$, where $Z = \sum_{i=1}^N \varphi(C_i)$.  Let $\mathcal{F}_j = \sigma( \mathscr{E}_j)$, where $\mathscr{E}_j$ is defined by \eqref{eq:AllRandomVectors}, set $\delta_j = \delta (k-j)/k$, $0 \leq j \leq k$, $\eta = \left( \epsilon^{-1} 4 e^{2/\epsilon} \max\{ 1, E[Z^{p+\delta}]\}  \right)^{-(p+\delta_j)/(p+\delta_{j+1})}$, and
\begin{equation} \label{eq:Xdefinition}  
Y_i^{(j,m)} = \left( \sum_{r=1}^{N_i^{(j+1)}} \varphi(C_{(i,r)}^{(j+1)}) \left| \hat F_{j,m}^{-1}(U_{(i,r)}^{(j+1)}) - F_j^{-1}(U_{(i,r)}^{(j+1)}) \right| \right)^{p+\delta_{j+1}},
\end{equation}
for $i = 1, \dots, m$. Then, on the event $\left\{ \sup_{m \geq n}  d_{p+\delta_j} (\hat F_{j,m}, F_{j})^{p+\delta_j} \leq \eta \right\}$,  we have
$$P\left( \left.\sup_{m \geq n} \frac{1}{m} \sum_{i=1}^m Y_i^{(j,m)} 1(Y_i^{(j,m)} \leq m/\log m)  > \epsilon \right| \mathcal{F}_j \right) \leq 2(n-1)^{-1/2}.$$
\end{lemma}

\begin{proof}
We start by noting that
\begin{align}
&P\left( \left. \sup_{m \geq n} \frac{1}{m} \sum_{i=1}^m Y_i^{(j,m)} 1( Y_i^{(j,m)} \leq m/\log m)  > \epsilon \right| \mathcal{F}_j \right) \notag \\
&\leq \sum_{m=n}^\infty P\left( \left. \frac{1}{m} \sum_{i=1}^m Y_i^{(j,m)} 1( Y_i^{(j,m)} \leq m/\log m)  > \epsilon \right| \mathcal{F}_j \right). \label{eq:Chernoff}
\end{align}

To bound each of the probabilities in \eqref{eq:Chernoff} use Chernoff's bound to obtain that
\begin{align*}
&P\left( \left. \frac{1}{m} \sum_{i=1}^m Y_i^{(j,m)} 1( Y_i^{(j,m)} \leq m/\log m)  > \epsilon \right| \mathcal{F}_j \right)  \\
&\leq \min_{\theta \geq 0} e^{-\theta \epsilon m } \left( E \left[ \left. e^{\theta Y_1^{(j,m)} 1(Y_1^{(j,m)} \leq m/\log m)} \right| \mathcal{F}_j \right] \right)^m.
\end{align*}

Note that by Remark~\ref{R.Assumptions}(i), we have that on the event  $\left\{ \sup_{m \geq n} d_{p+\delta_j} (\hat F_{j,m},  F_{j})^{p+\delta_j} \leq \eta \right\}$, 
\begin{align*}
E\left[ \left. Y_1^{(j,m)} \right| \mathcal{F}_j \right] &\leq 2 E[ Z^{p+\delta_{j+1}}] E\left[ \left. \left| \hat F_{j,m}^{-1}(U_1) - F_j^{-1}(U_1) \right|^{p+\delta_{j+1}} \right| \mathcal{F}_j \right] \\
&= || Z ||_{p+\delta_{j+1}}^{p+\delta_{j+1}} d_{p+\delta_{j+1}}(\hat F_{j,m}, F_{j})^{p+\delta_{j+1}} \\
&\leq  || Z ||_{p+\delta}^{p+\delta_{j+1}} d_{p+\delta_{j}}(\hat F_{j,m}, F_{j})^{p+\delta_{j+1}} \\
&\leq \max\{1, E[ Z^{p+\delta}] \}  \eta^{(p+\delta_{j+1})/(p+\delta_j)}  = \frac{\epsilon}{4 e^{2/\epsilon}}.
\end{align*}

Next, use the inequality $e^x \leq 1 + x e^x$ for $x \geq 0$  to obtain that
\begin{align*}
&E \left[ \left. e^{\theta Y_1^{(j,m)} 1(Y_1^{(j,m)} \leq m/\log m)} \right| \mathcal{F}_j \right] \\
&\leq 1 + \theta E \left[ \left. Y_1^{(j,m)} 1(Y_1^{(j,m)} \leq m/\log m) \, e^{\theta Y_1^{(j,m)} 1(Y_1^{(j,m)} \leq m/\log m)} \right| \mathcal{F}_j \right] \\
&\leq 1 + \theta E \left[ \left. Y_1^{(j,m)} \right| \mathcal{F}_j \right] e^{\theta  m/\log m} \\
&\leq 1 + \theta e^{\theta m/\log m} \frac{\epsilon}{4 e^{2/\epsilon}} . 
\end{align*}
Now use the inequality $1 + x \leq e^x$ to see that
$$\left( E \left[ \left. e^{\theta Y_1^{(j,m)} 1(Y_1^{(j,m)} \leq m/\log m)} \right| \mathcal{F}_j \right] \right)^m \leq e^{\theta \epsilon m e^{\theta m/\log m} /(4e^{2/\epsilon})}.$$
It follows that by choosing $\theta = (2/\epsilon) \log m/m$ we obtain
\begin{align*}
P\left( \left. \frac{1}{m} \sum_{i=1}^m Y_i^{(j,m)} 1(Y_i^{(j,m)} \leq m/\log m ) > \epsilon \right| \mathcal{F}_j \right) &\leq \min_{\theta \geq 0} e^{-\theta \epsilon m +\theta \epsilon m e^{\theta m/\log m} /(4e^{2/\epsilon}) } \\
&= \min_{\theta \geq 0} e^{-\theta \epsilon m \left(1 - \frac{e^{\theta m/\log m}  }{4 e^{2/\epsilon}}   \right)  }  \\
&\leq e^{-2\log m \left( 1 - \frac{1}{4 }   \right)  } ,
\end{align*}
which in turn implies that \eqref{eq:Chernoff} is bounded from above by
\begin{align*}
\sum_{m=n}^\infty e^{-(3/2) \log m  } &=  \sum_{m=n}^\infty m^{-3/2} \leq \sum_{m=n}^\infty \int_{m-1}^m \frac{1}{x^{3/2}} \, dx \\
&=  \int_{n-1}^\infty x^{-3/2} \, dx = 2(n-1)^{-1/2}.
\end{align*}
 This completes the proof.
\end{proof}

The next lemma gives the complementary estimate for the probability that any of the $\{Y_i^{(j,m)}\}$ exceeds the truncation value in Lemma~\ref{L.ConditionalSums}. The challenge here is the uniformity in $m$ of the result.

\begin{lemma} \label{L.CleverBound}
Fix $1 \leq p < \infty$. Suppose Assumption~\ref{A.Lipschitz} holds and $E[Z^{p+\delta}] < \infty$ for some $\delta > 0$, where $Z = \sum_{i=1}^N \varphi(C_i)$. Let $\delta_j = \delta (k-j)/k$ and $q_j = p + \delta_j$ for $0 \leq j < k$, fix $\eta > 0$, and let $Y_1^{(j,m)} $ be defined according to \eqref{eq:Xdefinition}. Then, for any $q_{j+1} < r_j < q_j$ and all $t \geq n$,
\begin{align*} 
&P\left(  \sup_{m \geq t} \frac{\log m}{m} Y_1^{(j,m)} > 1, \, \sup_{m \geq n}  d_{p+\delta_j} (\hat F_{j,m}, F_{j})^{p+\delta_j} \leq \eta \right) \\
&\leq  3^{r_j} ||Z||_{r_j}^{r_j}  \left\{ \frac{4 (\eta^{1/q_j}  + || R^{(j)} ||_{q_j} )^{r_j}  }{1- r_j/q_j} + 2 || R^{(j)} ||_{r_j}^{r_j}   \right\} \left( \frac{\log t}{t} \right)^{r_j/q_{j+1}}. 
\end{align*}
\end{lemma}

\begin{proof}
To simplify the notation, let 
$$( Q, N, \{C_r\}_{r \geq 1}, \{ U_r\}_{r \geq 1}) = \left( Q_1^{(j+1)}, N_1^{(j+1)}, \{ C_{(1,r)}^{(j+1)} \}_{r \geq 1}, \{ U_{(1,r)}^{(j+1)} \}_{r \geq 1} \right) .$$
 Next, note that
 \begin{align*}
&\sup_{m \geq t} \frac{\log m}{m} Y_1^{(j,m)} \\
&= \sup_{m \geq t} \frac{\log m}{m}  \left( \sum_{r=1}^N \varphi(C_r) \left| \hat F_{j,m}^{-1}(U_r) - F_j^{-1}(U_r) \right| \right)^{p+\delta_{j+1}} \\
&\leq   \left( \sum_{r=1}^N \varphi(C_r)  \sup_{m \geq t} \left( \frac{\log m}{m}  \right)^{1/(p+\delta_{j+1})} \left| \hat F_{j,m}^{-1}(U_r) - F_j^{-1}(U_r) \right| \right)^{p+\delta_{j+1}} \\
&= \left( \sum_{r=1}^N \varphi(C_r)  W_r^{(j,t)} \right)^{p+\delta_{j+1}},
\end{align*}
where
$$W_r^{(j,t)} = \sup_{m \geq t}  \left( \frac{\log m}{m}  \right)^{1/(p+\delta_{j+1})} \left| \hat F_{j,m}^{-1}(U_r) - F_j^{-1}(U_r) \right|.$$

Now, let $\mathcal{F}_j = \sigma(\mathscr{E}_j)$, where $\mathscr{E}_j$ is given by \eqref{eq:AllRandomVectors}, and note that
\begin{align*}
&P\left(  \sup_{m \geq t} \frac{\log m}{m} Y_1^{(j,m)} > 1, \, \sup_{m \geq n}  d_{p+\delta_j} (\hat F_{j,m}, F_{j})^{p+\delta_j} \leq \eta \right) \\
&\leq P\left(  \sum_{r=1}^N \varphi(C_r)  W_r^{(j,t)}  > 1, \, \sup_{m \geq n}  d_{p+\delta_j} (\hat F_{j,m}, F_{j})^{p+\delta_j} \leq \eta \right) \\
&= E\left[ P\left( \left. \sum_{r=1}^N \varphi(C_r)  W_r^{(j,t)}  > 1 \right| \mathcal{F}_j \right) 1\left( \sup_{m \geq n}  d_{p+\delta_j} (\hat F_{j,m}, F_{j})^{p+\delta_j} \leq \eta \right) \right].
\end{align*}
Moreover, if we let $q_j = p+\delta_j$ and use Lemma~\ref{L.InverseBound}, we obtain that, conditionally on $\mathcal{F}_j$, 
\begin{align*}
W_r^{(j,t)} &\leq  \sup_{m \geq t} \left(  \frac{\log m}{m}  \right)^{1/q_{j+1}} \left| \hat F_{j,m}^{-1}(U_r)  \right| + \sup_{m \geq t}  \left( \frac{\log m}{m}  \right)^{1/q_{j+1} } | F_j^{-1}(U_r) | \\
&\leq \sup_{m \geq t}  \left(  \frac{\log m}{m}  \right)^{1/q_{j+1}} \left( E\left[ \left. | \hat F_{j,m}^{-1}(U_r) |^{q_j} \right| \mathcal{F}_j \right] \right)^{1/q_j} \left\{ U_r^{-1/q_j} + (1-U_r)^{-1/q_j}  \right\}  \\
&\hspace{5mm} + \left(  \frac{\log t}{t}  \right)^{1/q_{j+1}}  | F_j^{-1}(U_r) |.
\end{align*}
Furthermore, by Minkowski's inequality, we have that on the event \linebreak $\{ \sup_{m \geq n} d_{q_j} (\hat F_{j,m}, F_j)^{q_j} \leq \eta\}$, 
\begin{align*}
&\sup_{m \geq t}  \left(  \frac{\log m}{m}  \right)^{1/q_{j+1}} \left( E\left[ \left. | \hat F_{j,m}^{-1}(U_r) |^{q_j} \right| \mathcal{F}_j \right] \right)^{1/q_j} \\
&\leq \sup_{m \geq t}  \left(  \frac{\log m}{m}  \right)^{1/q_{j+1}} \left\{ \left( E\left[ \left. | \hat F_{j,m}^{-1}(U_r) - F_j^{-1}(U_r) |^{q_j} \right| \mathcal{F}_j \right] \right)^{1/q_j}  + || F_j^{-1}(U_r) ||_{q_j} \right\} \\
&= \sup_{m \geq t}  \left(  \frac{\log m}{m}  \right)^{1/q_{j+1}} \left\{ d_{q_j}(\hat F_{j,m}, F_j)  + || R^{(j)} ||_{q_j} \right\} \\
&\leq  \left(  \frac{\log t}{t}  \right)^{1/q_{j+1}} \left\{ \eta^{1/q_j}  + || R^{(j)} ||_{q_j} \right\} . 
\end{align*}
It follows that conditionally on $\mathcal{F}_j$, we have that on the event \linebreak $\{ \sup_{m \geq n} d_{q_j} (\hat F_{j,m}, F_j)^{q_j} \leq \eta\}$,
$$W_r^{(j,t)} \leq \left( \frac{\log t}{t} \right)^{1/q_{j+1}} \left\{ K_j \left( U_r^{-1/q_j} + (1-U_r)^{-1/q_j} \right) + |F_j^{-1}(U_r)| \right\},$$
where $K_j \triangleq \eta^{1/q_j}  + || R^{(j)} ||_{q_j} < \infty$ by Remark~\ref{R.Assumptions}(ii).

Thus, we have that on the event $\{ \sup_{m \geq n} d_{q_j} (\hat F_{j,m}, F_j)^{q_j} \leq \eta\}$, the union bound and Markov's inequality yield
\begin{align*}
&P\left( \left. \sum_{r=1}^N \varphi(C_r)  W_r^{(j,t)}  > 1 \right| \mathcal{F}_j \right) \\
&\leq P\left( \sum_{r=1}^N \varphi(C_r)   \left\{ K_j \left( U_r^{-1/q_j} + (1-U_r)^{-1/q_j} \right) + |F_j^{-1}(U_r)| \right\}  > \left( \frac{t}{\log t} \right)^{1/q_{j+1}} \right) \\
&\leq P\left(  \sum_{r=1}^N \varphi(C_r)   K_j U_r^{-1/q_j}   > \frac{1}{3} \left( \frac{t}{\log t} \right)^{1/q_{j+1}}  \right) \\
&\hspace{5mm} + P\left(  \sum_{r=1}^N \varphi(C_r)   K_j (1-U_r)^{-1/q_j} > \frac{1}{3} \left( \frac{t}{\log t} \right)^{1/q_{j+1}} \right) \\
&\hspace{5mm} + P\left(  \sum_{r=1}^N \varphi(C_r)   |F_j^{-1}(U_r)|   > \frac{1}{3} \left( \frac{t}{\log t} \right)^{1/q_{j+1}}  \right) \\
&\leq 3^{r_j} \left( \frac{\log t}{t} \right)^{r_j/q_{j+1}} \left\{ 2 E\left[ \left( \sum_{i=1}^N \varphi(C_i) K_j U_i^{-1/q_j} \right)^{r_j} \right] \right. \\
&\hspace{5mm} \left. + E\left[ \left( \sum_{i=1}^N \varphi(C_i) R_i^{(j)} \right)^{r_j} \right]   \right\} ,
\end{align*}
where by assumption $q_{j+1} < r_j < q_j$, and we have used the observation that $U_i \stackrel{\mathcal{D}}{=} 1- U_i$. Finally, note that by Remark~\ref{R.Assumptions}(i), we have
$$E\left[ \left( \sum_{i=1}^N \varphi(C_i) K_j U_i^{-1/q_j} \right)^{r_j} \right] \leq 2 E[ Z^{r_j} ] K_j^{r_j} E[ U_1^{-r_j/q_j}]  = \frac{2 K_j^{r_j} || Z ||_{r_j}^{r_j}}{1 - r_j/q_j} $$
and
$$E\left[ \left( \sum_{i=1}^N \varphi(C_i) R_i^{(j)} \right)^{r_j} \right] \leq 2 E[Z^{r_j}] E\left[ |R^{(j)}|^{r_j} \right] = 2 || Z ||_{r_j}^{r_j} || R^{(j)} ||_{r_j}^{r_j}.$$

We conclude that
\begin{align*}
&P\left(  \sup_{m \geq t} \frac{\log m}{m} X_1^{(m)} > 1, \, \sup_{m \geq n}  d_{p+\delta_j} (\hat F_{j,m}, F_{j})^{p+\delta_j} \leq \eta \right) \\
&\leq 3^{r_j} \left( \frac{\log t}{t} \right)^{r_j/q_{j+1}} \left\{ \frac{4 K_j^{r_j} ||Z||_{r_j}^{r_j} }{1- r_j/q_j} + 2 || Z ||_{r_j}^{r_j} || R^{(j)} ||_{r_j}^{r_j}   \right\}.
\end{align*}
\end{proof}

\bigskip

We are now ready to prove Theorem~\ref{T.AlmostSure2}, which proves by induction that $d_{p+\delta}(\hat F_{k,m}, F_k) \to 0$ a.s. as $m \to \infty$.

\bigskip

\begin{proof}[Proof of Theorem~\ref{T.AlmostSure2}]
Define $\delta_j = \delta (k-j)/k$ for $0 \leq j \leq k$. We will prove by induction in $j$ that 
\begin{equation} \label{eq:InductionHypothesis}
\lim_{m \to \infty}  d_{p+\delta_j}(\hat F_{j,m}, F_j) = 0 \qquad \text{a.s}
\end{equation}
for $0 \leq j \leq k$. Since $\hat F_{0,m}(x) \equiv F_{0,m}(x)$ for all $x \in \mathbb{R}$ and $E[|R_0|^{p+\delta}] < \infty$,  the Glivenko-Cantelli lemma and the strong law of large numbers yield
\begin{align*} 
\sup_{x \in \mathbb{R}} \left| F_{0,m}(x) - F_0(x) \right| &\to 0 \qquad \text{a.s. as $m \to \infty$} \qquad \text{ and } \\
 \frac{1}{m} \sum_{i=1}^m |R^{(0)}_i |^{p+\delta} = \int_{-\infty}^\infty |x|^{p+\delta} dF_{0,m}(x) &\to \int_{-\infty}^\infty |x|^{p+\delta} dF_0(x) \quad \text{a.s. as $m \to \infty$}.
 \end{align*}
Therefore, by Definition~6.8 and Theorem~6.9 in \cite{Villani_2009}, 
$$\lim_{m \to \infty} d_{p+\delta_0} (\hat F_{0,m}, F_0) = \lim_{m \to \infty} d_{p+\delta}(F_{0,m}, F_0) = 0 \quad \text{a.s.}$$

Suppose now that \eqref{eq:InductionHypothesis} holds for $0 \leq j < k$.  To prove that \linebreak $d_{p+\delta_{j+1}}(\hat F_{j+1,m}, F_{j+1}) \to 0$ a.s.~as $m \to \infty$, we start by constructing the random variables $\{ (\hat R_i^{(t,m)}, R_i^{(t)}) : 1 \leq i \leq m, \, 0 \leq t \leq k \}$ as explained at the beginning of this section. Now note that for any $\epsilon, \eta > 0$, 
\begin{align}
&P\left(  \sup_{m \geq n}  d_{p+\delta_{j+1}} (\hat F_{j+1,m}, F_{j+1})^{p+\delta_{j+1}} > 2^{p+\delta_{j+1}} \epsilon \right) \notag \\
&\leq P\left( \sup_{m \geq n}  \left\{ d_{p+\delta_{j+1}} (\hat F_{j+1,m}, F_{j+1,m}) + d_{p+\delta_{j+1}} (F_{j+1,m}, F_{j+1})  \right\} > 2 \epsilon^{1/(p+\delta_{j+1})} \right)
 \notag  \\
 &\leq P\left( \sup_{m \geq n} d_{p+\delta_{j+1}} (\hat F_{j+1,m}, F_{j+1,m}) > \epsilon^{1/(p+\delta_{j+1})} \right) \notag \\
 &\hspace{5mm} + P\left( \sup_{m \geq n} d_{p+\delta_{j+1}} ( F_{j+1,m}, F_{j+1}) > \epsilon^{1/(p+\delta_{j+1})} \right)\notag \\
 &\leq P\left( \sup_{m \geq n} d_{p+\delta_{j+1}} ( \hat F_{j+1,m}, F_{j+1,m})^{p+\delta_{j+1}} > \epsilon, \, \sup_{m \geq n}  d_{p+\delta_j} (\hat F_{j,m}, F_{j})^{p+\delta_j} \leq \eta   \right) \label{eq:ExponentialTerm} \\
 &\hspace{5mm} + P\left( \sup_{m \geq n}  d_{p+\delta_j} (\hat F_{j,m},  F_{j})^{p+\delta_j} > \eta \right) \label{eq:InductionStep} \\
  &\hspace{5mm} + P\left( \sup_{m \geq n}  d_{p+\delta_{j+1}} (F_{j+1,m}, F_{j+1})^{p+\delta_{j+1}} > \epsilon \right) .   \label{eq:LLNStep}
 \end{align}

To analyze \eqref{eq:InductionStep} note that its convergence to zero as $n \to \infty$ is equivalent to the a.s.~convergence of $d_{p+\delta_j} (\hat F_{j,m},  F_{j})$ to zero as $m \to \infty$, which corresponds to the induction hypothesis \eqref{eq:InductionHypothesis}.

To show that \eqref{eq:LLNStep} converges to zero as $n \to \infty$, note that by Remark~\ref{R.Assumptions}(ii) we have $E[ |R^{(j+1)}|^{p+\delta} ] < \infty$, which implies that $E[ |R^{(j+1)}|^{p+\delta_{j+1}} ] < \infty$. Hence, the Glivenko-Cantelli lemma, the strong law of large numbers, and Definition~6.8 and Theorem~6.9 in \cite{Villani_2009} give that $\lim_{m \to \infty} d_{p+\delta_{j+1}}(F_{j+1,m}, F_{j+1}) = 0$ a.s., which is equivalent to
\begin{equation*} 
 \lim_{n \to \infty} P\left( \sup_{m \geq n}  d_{p+\delta_{j+1}} (F_{j+1,m}, F_{j+1})^{p+ \delta_{j+1}} > \epsilon \right) = 0.
\end{equation*}

Next, to prove that \eqref{eq:ExponentialTerm}  converges to zero we first define the random variables $\{ Y_i^{(j,m)}: 1 \leq i \leq m\}$ according to \eqref{eq:Xdefinition}, and define the events
$$A_{i,n} = \left\{ \sup_{m \geq n \vee i} \frac{\log m}{m}  Y_i^{(j,m)} \leq 1 \right\}.$$
Now use \eqref{eq:ExplicitBound} and Assumption~\ref{A.Lipschitz} to obtain
\begin{align}
&P\left( \sup_{m \geq n} d_{p+\delta_{j+1}} ( \hat F_{j+1,m}, F_{j+1,m})^{p+\delta_{j+1}} > \epsilon, \, \sup_{m \geq n}  d_{p+\delta_j} (\hat F_{j,m}, F_{j})^{p+\delta_j} \leq \eta   \right) \notag \\
&\leq P\left( \sup_{m \geq n} \frac{1}{m} \sum_{i=1}^m \left| \hat R_i^{(j+1,m)} - R_i^{(j+1)} \right|^{p+\delta_{j+1}}  > \epsilon, \, \sup_{m \geq n}  d_{p+\delta_j} (\hat F_{j,m}, F_{j})^{p+\delta_j} \leq \eta   \right) \notag \\
&\leq P\left( \sup_{m \geq n} \frac{1}{m} \sum_{i=1}^m Y_i^{(j,m)}  > \epsilon, \, \sup_{m \geq n}  d_{p+\delta_j} (\hat F_{j,m}, F_{j})^{p+\delta_j} \leq \eta, \, \bigcap_{i=1}^\infty A_{n,i}   \right) \notag \\
&\hspace{5mm} + P\left(  \sup_{m \geq n}  d_{p+\delta_j} (\hat F_{j,m}, F_{j})^{p+\delta_j} \leq \eta, \, \bigcup_{i=1}^\infty A_{n,i}^c   \right) \notag \\
&\leq P\left( \sup_{m \geq n} \frac{1}{m} \sum_{i=1}^m Y_i^{(j,m)} 1(Y_i^{(j,m)} \leq m/\log m)  > \epsilon, \, \sup_{m \geq n}  d_{p+\delta_j} (\hat F_{j,m}, F_{j})^{p+\delta_j} \leq \eta  \right) \label{eq:Chernoff} \\
&\hspace{5mm} + \sum_{i=1}^\infty P\left(  A_{n,i}^c, \, \sup_{m \geq n}  d_{p+\delta_j} (\hat F_{j,m}, F_{j})^{p+\delta_j} \leq \eta   \right). \label{eq:Maximum}
\end{align}

To analyze \eqref{eq:Chernoff}, choose $\eta = \left( \epsilon^{-1} 4 e^{2/\epsilon} \max\{ 1, E[Z^{p+\delta}]\}  \right)^{-(p+\delta_j)/(p+\delta_{j+1})}$ and let $\mathcal{F}_j = \sigma( \mathscr{E}_j)$ denote the sigma-algebra generated by $\mathscr{E}_j$, as given by \eqref{eq:AllRandomVectors}. Note that
\begin{align*}
&P\left( \sup_{m \geq n} \frac{1}{m} \sum_{i=1}^m Y_i^{(j,m)} 1(Y_i^{(j,m)} \leq m/\log m)  > \epsilon, \, \sup_{m \geq n}  d_{p+\delta_j} (\hat F_{j,m}, F_{j})^{p+\delta_j} \leq \eta  \right) \\
&= E\left[ P\left( \left.\sup_{m \geq n} \frac{1}{m} \sum_{i=1}^m Y_i^{(j,m)} 1(Y_i^{(j,m)} \leq m/\log m)  > \epsilon \right| \mathcal{F}_j \right) \right. \\
&\hspace{15mm} \left.  \cdot1\left( \sup_{m \geq n}  d_{p+\delta_j}(\hat F_{j,m}, F_{j})^{p+\delta_j} \leq \eta   \right) \right] .
\end{align*}

By Lemma~\ref{L.ConditionalSums}, we obtain that on the event $\left\{  \sup_{m \geq n}  d_{p+\delta_j}(\hat F_{j,m}, F_{j})^{p+\delta_j} \leq \eta \right\}$, we have
$$P\left( \left.\sup_{m \geq n} \frac{1}{m} \sum_{i=1}^m Y_i^{(j,m)} 1(Y_i^{(j,m)} \leq m/\log m)  > \epsilon \right| \mathcal{F}_j \right)  \leq 2 (n-1)^{-1/2},$$
which implies that \eqref{eq:Chernoff} is bounded from above by $2(n-1)^{-1/2}$. 

To analyze \eqref{eq:Maximum} note that
\begin{align*}
&\sum_{i=1}^\infty P\left(  A_{n,i}^c, \, \sup_{m \geq n}  d_{p+\delta_j} (\hat F_{j,m}, F_{j})^{p+\delta_j} \leq \eta   \right) \\
&= n P\left(  \sup_{m \geq n} \frac{\log m}{m} Y_1^{(j,m)} > 1, \, \sup_{m \geq n}  d_{p+\delta_j} (\hat F_{j,m}, F_{j})^{p+\delta_j} \leq \eta \right) \\
&\hspace{5mm} + \sum_{t=n+1}^\infty P\left(  \sup_{m \geq t} \frac{\log m}{m} Y_1^{(j,m)} > 1, \, \sup_{m \geq n}  d_{p+\delta_j} (\hat F_{j,m}, F_{j})^{p+\delta_j} \leq \eta \right).
\end{align*}
Now set $q_j = p + \delta_j$ and $r_j = q_{j+1} + \delta/(2k)$, and note that $q_{j+1} < r_j < q_j \leq p+\delta$. Then, by Lemma~\ref{L.CleverBound}, 
\begin{align*}
P\left(  \sup_{m \geq t} \frac{\log m}{m} Y_1^{(j,m)} > 1, \, \sup_{m \geq n}  d_{p+\delta_j} (\hat F_{j,m}, F_{j})^{p+\delta_j} \leq \eta \right) \leq \tilde K_j \left( \frac{\log t}{t} \right)^{r_j/q_{j+1}} 
\end{align*}
for any $t \geq n$, where
$$\tilde K_j =  3^{r_j} \left|\left|  \sum_{i =1}^N \varphi(C_i) \right|\right|_{r_j}^{r_j} \left\{ \frac{4 (\eta^{1/q_j}  + || R^{(j)} ||_{q_j} )^{r_j}  }{1- r_j/q_j} + 2 || R^{(j)} ||_{r_j}^{r_j}   \right\} < \infty$$
by Remark~\ref{R.Assumptions}(ii). It follows that \eqref{eq:Maximum} is bounded from above by
\begin{align*}
&\tilde K_j n \left( \frac{\log n}{n} \right)^{r_j/q_{j+1}} + \tilde K_j \sum_{t=n+1}^\infty \left( \frac{\log t}{t} \right)^{r_j/q_{j+1}} \\
&\leq  \tilde K_j n \left( \frac{\log n}{n} \right)^{r_j/q_{j+1}} + \tilde K_j \sum_{t=n+1}^\infty \int_{t-1}^t \left( \frac{\log x}{x} \right)^{r_j/q_{j+1}} \, dx \\
&= \tilde K_j n \left( \frac{\log n}{n} \right)^{r_j/q_{j+1}} + \tilde K_j  \int_{n}^\infty \left( \frac{\log x}{x} \right)^{r_j/q_{j+1}} \, dx
\end{align*}
for all $n \geq 3$. Since $r_j/q_{j+1} > 1$ and
$$\int_{n}^\infty \left( \frac{\log x}{x} \right)^{r_j/q_{j+1}} = \frac{(\log n)^{r_j/q_{j+1}}}{(r_j/q_{j+1} - 1) n^{r_j/q_{j+1} - 1}} (1 + o(1))$$
as $n \to \infty$, we conclude that \eqref{eq:ExponentialTerm}  is bounded from above by
$$2(n-1)^{-1/2} + \tilde K_j \left( 1 + \frac{1}{r_j/q_{j+1} - 1} + o(1) \right) \frac{(\log n)^{r_j/q_{j+1}} }{n^{r_j/q_{j+1}-1}},$$
which converges to zero as $n \to \infty$. This completes the proof. 
\end{proof}

\bigskip

We now give below the proof of Proposition~\ref{P.Consistency}.

\bigskip

\begin{proof}[Proof of Proposition~\ref{P.Consistency}]
The second statement of the proposition, regarding the almost sure convergence, follows directly from Definition~6.8 and Theorem~6.9 in \cite{Villani_2009}. For the convergence in probability we argue as follows.

Define $\Theta_{k,m} = \frac{1}{m} \sum_{i=1}^m h(\hat R_i^{(k,m)})$ and $\theta_k = E[h(R^{(k)})]$. By assumption, we have that $d_p(\hat F_{k,m}, F_k) \to 0$ in $L_p$ and therefore in probability, as $m \to \infty$. Hence, for every subsequence $\{m_i\}_{i \geq 1}$ there is a further subsequence $\{ m_{i_j} \}_{j \geq 1}$ such that $d_p(\hat F_{k,m_{i_j}}, F_k) \to 0$ a.s. as $j \to \infty$. Definition~6.8 and Theorem~6.9 in \cite{Villani_2009} now give that
\begin{equation} \label{eq:Subsequence}
\Theta_{k,m_{i_j}} \to \theta_k \quad \text{a.s. as } j \to \infty.
\end{equation}
We conclude that for any subsequence $\{m_i\}_{i \geq 1}$ we can find a further subsequence $\{m_{i_j}\}_{j \geq 1}$ such that \eqref{eq:Subsequence} holds, and therefore,
$$\Theta_{k,m} \xrightarrow{P} \theta_k \quad \text{ as } m \to \infty.$$
\end{proof}

The remaining two proofs in the paper correspond to Theorem~\ref{T.LpConvergence} and Lemma~\ref{L.Moments}, which although not directly related to the Population Dynamics algorithm, may be of independent interest.

\begin{proof}[Proof of Theorem~\ref{T.LpConvergence}]
Suppose first that Assumption~\ref{A.PhiAssumption} holds for any i.i.d.~$\{(X_i, Y_i): i \geq 1\}$ independent of $(Q, N, \{C_i\})$.  Recall that $F_k(x) = P(R^{(k)} \leq x)$. Then, for any $j \in \mathbb{N}_+$ we have
\begin{align}
d_p(F_{j}, F_{j-1}) \notag &\leq \left(  E\left[  \left| \Phi(Q, N, \{C_r\}, \{F_{j-1}^{-1}(U_r) \}) - \Phi(Q, N, \{C_r\}, \{F_{j-2}^{-1}(U_r) \})  \right|^p \right] \right)^{1/p} \notag \\
&\leq H_p^{1/p} \left( E\left[ \left| F_{j-1}^{-1}(U_1) - F_{j-2}^{-1}(U_1) \right|^p  \right] \right)^{1/p} \notag \\
&= H_p^{1/p} d_p(F_{j-1}, F_{j-2}) \notag \\
&\leq (H_p^{1/p})^{j-1} d_p(F_1, F_0). \label{eq:Contraction}
\end{align}
Moreover, 
\begin{align}
d_p(F_1, F_0) &\leq \left(  E\left[  \left| \Phi(Q, N, \{C_r\}, \{F_{0}^{-1}(U_r) \}) - \Phi(Q, N, \{C_r\}, \{ 0 \})  \right|^p \right] \right)^{1/p} \notag \\
&\hspace{5mm} + \left(  E\left[  \left|  \Phi(Q, N, \{C_r\}, \{ 0 \})  \right|^p \right] \right)^{1/p} \notag \\
&\leq H_p^{1/p} \left( E[ |R^{(0)} |^p] \right)^{1/p} + \left(  E\left[  \left|  \Phi(Q, N, \{C_r\}, \{ 0 \})  \right|^p \right] \right)^{1/p}. \label{eq:FirstStepMoment}
\end{align}
It follows that for any $m \in \mathbb{N}_+$ we have
\begin{align*}
d_p(F_{k+m}, F_k) &\leq \sum_{j=1}^m d_p(F_{k+j}, F_{k+j-1}) \leq  \sum_{j=1}^m (H_p^{1/p})^{k+j-1} d_p(F_1, F_0) \\
&\leq (H_p^{1/p})^{k} d_p(F_1, F_0) \sum_{j=0}^{m-1} (H_p^{1/p})^{j},
\end{align*}
which converges to zero as $k \to \infty$ uniformly in $m$ whenever $H_p < 1$ and $E\left[ |R_0|^p + | \Phi(Q, N, \{C_r\}, \{ 0\})|^p \right]$. Therefore, the sequence $\{ R^{(k)}: k \geq 0\}$ is Cauchy, and since the Wasserstein space $P_p(\mathbb{R})$ metrized by $d_p$ (see Definition~6.4 in \cite{Villani_2009}) is complete by Theorem~6.18 in \cite{Villani_2009}, we have that there exists a random variable $R$ having distribution $F_*(x) = P(R \leq x)$ such that
$$\lim_{k \to \infty} d_p(F_k, F_*) = 0.$$
Equation \eqref{eq:GeomConv} now follows by taking $m \to \infty$ to obtain:
$$d_p(F_k, F_*)^p = \lim_{m \to \infty} d_p(F_k, F_{k+m})^p \leq d_p(F_1, F_0)^p \frac{H_p^{k}}{(1 - H_p^{1/p})^p}$$
and using the optimal coupling $(R^{(k)}, R) = (F_k^{-1}(U), F_*^{-1}(U))$. \\

We now move to the linear SFPE \eqref{eq:Linear}, for which it is known (see \cite{Jel_Olv_12b}) that $R$ admits the explicit representation
$$R = \sum_{k=0}^\infty \sum_{{\bf i} \in A_k} \Pi_{\bf i} Q_{\bf i},$$
as described in Section~\ref{SS.EndogenousSol}. When conditions (i) hold we  have $E[ R^{(k)} ] = 0$ for all $k \geq 0$ and the arguments used above remain valid. 

Suppose now that conditions (ii) hold, in which case we can take $R^{(k)} = \sum_{j=0}^{k-1} \sum_{{\bf i} \in A_j} \Pi_{\bf i} Q_{\bf i} + \sum_{{\bf i} \in A_k} \Pi_{\bf i} R^{(0)}_{\bf i}$, where the $\{ R^{(0)}_{\bf i}: {\bf i} \in U\}$ are i.i.d.~copies of $R^{(0)}$. Therefore, Minkowski's inequality gives
\begin{align*}
E\left[ | R^{(k)} - R |^p \right] &\leq E\left[ \left( \sum_{j=k}^\infty \sum_{{\bf i} \in A_j}  |\Pi_{\bf i}| |Q_{\bf i}| + \sum_{{\bf i} \in A_j}  |\Pi_{\bf i}| |R^{(0)}_{\bf i}| \right)^p \right] \\
&\leq \left( \sum_{j=k}^\infty \left( E\left[ \left(  W_j \right)^p \right] \right)^{1/p} +  \left( E\left[ \left(  W_k(R^{(0)}) \right)^p \right] \right)^{1/p}  \right)^p,
\end{align*}
where $W_j \triangleq  \sum_{{\bf i} \in A_j}  |\Pi_{\bf i}| |Q_{\bf i}|$ and $W_k(R^{(0)}) \triangleq \sum_{{\bf i} \in A_k}  |\Pi_{\bf i}| |R^{(0)}_{\bf i}|$. Now use Lemma~4.4 in \cite{Jel_Olv_12b} to obtain that under conditions (ii) there exist a constants $K_p, K_p' < \infty$ such that
$$E[ |W_j|^p ] \leq K_p ( \rho_1 \vee \rho_p )^j \qquad \text{and} \qquad E[|W_k(R^{(0)} )|^p ] \leq K_p (\rho_1 \vee \rho_p)^k,$$
where $\rho_\beta \triangleq E\left[ \sum_{i=1}^N |C_i|^\beta \right]$. Hence,
\begin{align*}
E\left[ | R^{(k)} - R |^p \right] &\leq \left( (K_p + K_p') \sum_{j=k-1}^\infty (\rho_1 \vee \rho_p)^{j/p}  \right)^p \\
&\leq \left( \frac{K_p+K_p'}{1 - (\rho_1 \vee \rho_p)^{1/p}} \right)^p (\rho_1 \vee \rho_p)^{k-1}.
\end{align*}
This completes the proof.
\end{proof}

Finally, we provide the proof of Lemma~\ref{L.Moments}.

\begin{proof}[Proof of Lemma~\ref{L.Moments}]
By \eqref{eq:Contraction} we have for any $j \in \mathbb{N}_+$, 
$$d_p(F_j, F_{j-1}) \leq (H_p^{1/p})^{j-1} d_p(F_1, F_0),$$
and by \eqref{eq:FirstStepMoment}, 
$$d_p(F_1, F_0) \leq H_p^{1/p} \left( E[ |R^{(0)}|^p] \right)^{1/p} + \left(  E\left[  \left|  \Phi(Q, N, \{C_r\}, \{ 0 \})  \right|^p \right] \right)^{1/p} \triangleq A_p'. $$
Hence,
$$d_p(F_k, F_0) \leq \sum_{i=1}^k d_p(F_i, F_{i-1}) \leq A_p' \sum_{i=1}^k (H_p^{1/p})^{i-1},$$
and we obtain that
\begin{align*}
\left( E\left[ |R^{(k)}|^p \right] \right)^{1/p} &\leq \left( E\left[ |F_k^{-1}(U) - F_0^{-1}(U) |^p \right] \right)^{1/p} + \left( E\left[ |R^{(0)}|^p \right] \right)^{1/p} \\
&= d_p(F_k, F_0) + \left( E\left[ |R^{(0)}|^p \right] \right)^{1/p} \\
&\leq A_p' \sum_{i=1}^k (H_p^{1/p})^{i-1}  + \left( E\left[ |R^{(0)}|^p \right] \right)^{1/p} \\
&\leq \left(A_p' + \left( E\left[ |R^{(0)}|^p \right] \right)^{1/p} \right) \sum_{i=0}^{k-1} (H_p^{1/p})^i  . 
\end{align*}
\end{proof}

\bibliographystyle{plain}
\bibliography{PopulationDynamicsBib}

\begin{thebibliography}{10}

\bibitem{Aldo_Band_05}
D.J. Aldous and A.~Bandyopadhyay.
\newblock A survey of max-type recursive distributional equation.
\newblock {\em Annals of Applied Probability}, 15(2):1047--1110, 2005.

\bibitem{Als_Big_Mei_10}
G.~Alsmeyer, J.D. Biggins, and M.~Meiners.
\newblock The functional equation of the smoothing transform.
\newblock {\em Ann. Probab.}, 40(5):2069--2105, 2012.

\bibitem{Als_Dysz_17}
G.~Alsmeyer and P.~Dyszewski.
\newblock Thin tails of fixed points of the nonhomogeneous smoothing transform.
\newblock {\em Stochastic Processes and their Applications}, 2017.

\bibitem{Alsm_Mein_10a}
G.~Alsmeyer and M.~Meiners.
\newblock Fixed points of inhomogeneous smoothing transforms.
\newblock {\em J. Differ. Equ. Appl.}, 18(8):1287--1304, 2012.

\bibitem{Alsm_Mein_10b}
G.~Alsmeyer and M.~Meiners.
\newblock Fixed points of the smoothing transform: {T}wo-sided solutions.
\newblock {\em Probab. Theory Rel.}, 155(1-2):165--199, 2013.

\bibitem{Athreya_85}
K.B. Athreya.
\newblock Discounted branching random walks.
\newblock {\em Advances in Applied Probability}, 17:53--66, 1985.

\bibitem{Biggins_98}
J.D. Biggins.
\newblock Lindley-type equations in the branching random walk.
\newblock {\em Stochastic Process. Appl.}, 75:105--133, 1998.

\bibitem{Bob_Led_17}
S.~Bobkov and M.~Ledoux.
\newblock One-dimensional empirical measures, order statistics, and
  {K}antorovich transport distances.
\newblock {\em To appear in Memoirs of the American Mathematical Society},
  2017.

\bibitem{Chen_Litv_Olv_14}
N.~Chen, N.~Litvak, and M.~Olvera-Cravioto.
\newblock Generalized pagerank on directed configuration networks.
\newblock {\em Random Structures \& Algorithms}, 2016.

\bibitem{Chen_Olv_15}
N.~Chen and M.~Olvera-Cravioto.
\newblock Efficient simulation for branching linear recursions.
\newblock In {\em Proceedings of the Winter Simulation Conference 2015}, pages
  2716--2727, Huntington Beach, CA, 2015.

\bibitem{Barr_Gin_Mat_99}
E.~del Barrio, E.~Gin{\'e}, and C.~Matr{\'a}n.
\newblock Central limit theorems for the {W}asserstein distance between the
  empirical and the true distributions.
\newblock {\em Annals of Probability}, pages 1009--1071, 1999.

\bibitem{Dem_Mon_10}
A.~Dembo and A.~Montanari.
\newblock Ising models on locally tree-like graphs.
\newblock {\em Ann. Appl. Probab.}, 20(2):565--592, 2010.

\bibitem{Devroye_01}
L.~Devroye.
\newblock On the probabilistic worst-case time of {FIND}.
\newblock {\em Algorithmica}, 31:291--303, 2001.

\bibitem{Fill_Jan_01}
J.A. Fill and S.~Janson.
\newblock Approximating the limiting {Q}uicksort distribution.
\newblock {\em Random Structures Algorithms}, 19(3-4):376--406, 2001.

\bibitem{Fou_Gui_15}
N.~Fournier and A.~Guillin.
\newblock On the rate of convergence in {W}asserstein distance of the empirical
  measure.
\newblock {\em Probab. Theory Relat. Fields}, 162:707--738, 2015.

\bibitem{Iksanov_04}
A.M. Iksanov.
\newblock Elementary fixed points of the {BRW} smoothing transforms with
  infinite number of summands.
\newblock {\em Stochastic Process. Appl.}, 114:27--50, 2004.

\bibitem{Janson_15}
S.~Janson.
\newblock On the tails of the limiting quicksort distribution.
\newblock {\em Electron. Commun. Probab.}, 20(81):1--7, 2015.

\bibitem{Jel_Olv_10}
P.R. Jelenkovi\'c and M.~Olvera-Cravioto.
\newblock Information ranking and power laws on trees.
\newblock {\em Adv. Appl. Prob.}, 42(4):1057--1093, 2010.

\bibitem{Jel_Olv_12b}
P.R. Jelenkovi\'c and M.~Olvera-Cravioto.
\newblock Implicit renewal theorem for trees with general weights.
\newblock {\em Stochastic Process. Appl.}, 122(9):3209--3238, 2012.

\bibitem{Jel_Olv_12a}
P.R. Jelenkovi\'c and M.~Olvera-Cravioto.
\newblock Implicit renewal theory and power tails on trees.
\newblock {\em Adv. Appl. Prob.}, 44(2):528--561, 2012.

\bibitem{Jel_Olv_15}
P.R. Jelenkovi\'c and M.~Olvera-Cravioto.
\newblock Maximums on trees.
\newblock {\em Stochastic Processes and their Applications}, 125:217--232,
  2015.

\bibitem{Kar_Kel_Suh_94}
F.I. Karpelevich, M.Ya. Kelbert, and Yu.M. Suhov.
\newblock Higher-order {L}indley equations.
\newblock {\em Stochastic Process. Appl.}, 53:65--96, 1994.

\bibitem{Lee_Olv_17}
J.~Lee and M.~Olvera-Cravioto.
\newblock Pagerank on inhomogeneous random digraphs.
\newblock {\em ArXiv:1707.02492}, 2017.

\bibitem{Liu_98}
Q.~Liu.
\newblock Fixed points of a generalized smoothing transformation and
  applications to the branching random walk.
\newblock {\em Adv. Appl. Prob.}, 30:85--112, 1998.

\bibitem{Mezard_Montanari_2009}
M.~Mezard and A.~Montanari.
\newblock {\em Information, physics, and computation}.
\newblock Oxford University Press, 2009.

\bibitem{Olv_Ruiz_14}
M.~Olvera-Cravioto and O.~Ruiz-Lacedelli.
\newblock Parallel queues with synchronization.
\newblock {\em arXiv:1501.00186}, 2014.

\bibitem{Rosler_91}
U.~R\"{o}sler.
\newblock A limit theorem for {``Quicksort"}.
\newblock {\em RAIRO Theor. Inform. Appl.}, 25:85--100, 1991.

\bibitem{Ros_Rus_01}
U.~R\"{o}sler and L.~R\"{u}schendorf.
\newblock The contraction method for recursive algorithms.
\newblock {\em Algorithmica}, 29(1-2):3--33, 2001.

\bibitem{Villani_2009}
C.~Villani.
\newblock {\em Optimal transport, old and new}.
\newblock Springer, New York, 2009.

\bibitem{Volk_Litv_08}
Y.~Volkovich and N.~Litvak.
\newblock Asymptotic analysis for personalized web search.
\newblock {\em Adv. Appl. Prob.}, 42(2):577--604, 2010.

\end{thebibliography}

\end{document}